\documentclass[12pt,a4paper,reqno]{amsart}
\usepackage{amsmath}
\date{}
 \usepackage[latin1]{inputenc}
\usepackage{amssymb , amsmath, amsthm}
 \usepackage{latexsym, esint}
 \usepackage{ulem}

\usepackage{latexsym}

\usepackage{color}

\newtheorem{thm}{Theorem}[section]
\newtheorem{ex}{Example}[section]
\newtheorem{rem}{Remark}[section]
 \newtheorem{prop}{Proposition}[section]
\newtheorem{lemme}{Lemma}[section]
\newtheorem{cor}{Corollary}[section]

\newcommand{\Sp}{\mathbb S}

\newcommand{\R}{\mathbb R}
\newcommand{\N}{\mathbb N}

\DeclareMathOperator*{\essinf}{ess\,inf}
 
\begin{document}

\title[Spectrum of the Laplacian with weights]
{Spectrum of the Laplacian with weights }

\author{Bruno Colbois}
\address{ Universit\'e de Neuch\^atel, Laboratoire de
Math\'ematiques, 13 rue E. Argand, 2007 Neuch\^atel, Switzerland.}
\email{Bruno.Colbois@unine.ch}

\author{Ahmad El Soufi}
\address{Universit\'e de Tours, Laboratoire de Math\'ematiques
et Physique Th\'eorique, UMR-CNRS 7350, Parc de Grandmont, 37200
Tours, France.} \email{elsoufi@univ-tours.fr.}

\thanks{}

\begin{abstract}
Given a compact Riemannian manifold $(M,g)$ and two positive functions $\rho$ and $\sigma$, we are interested in the eigenvalues of  the Dirichlet energy functional weighted by $\sigma$,  with respect to the $L^2$ inner product  weighted by $\rho$. Under some regularity conditions on $\rho$ and $\sigma$, these eigenvalues are those of the operator
$
-\rho^{-1} \mbox{div}(\sigma \nabla u)
$
with  Neumann conditions on the boundary if $\partial M\ne \emptyset$.
We investigate the effect of the weights on eigenvalues and
discuss the existence of lower and upper bounds under the condition that the total mass is preserved.  

\end{abstract}


\maketitle

\section{Introduction}

Let $(M,g)$ be a compact Riemannian manifold of dimension $n\ge 2$, possibly with  nonempty boundary. We designate by $\left\{\lambda_k(M,g)\right\}_{k\ge 0}$ the nondecreasing sequence of eigenvalues of the Laplacian on $(M,g)$ under  Neumann conditions on the boundary   if $\partial M\ne \emptyset$.  The min-max principle tells us that these eigenvalues are variationally defined by
$$\lambda_k(M,g)=\inf_{E\in S_{k+1}} \sup_{u\in E\setminus \{ 0\} } \frac{\int_M |\nabla u|^2 v_g}{\int_M  u^2  v_g} $$
where $S_k$ is the set of all $k$-dimensional vector subspaces of  $H^1(M)$ and $v_g$ is the Riemannian volume element associated with  $ g $.  

\smallskip
The relationships between the eigenvalues $\lambda_k(M,g)$ and the other geometric data of $(M,g)$ constitute a classical topic of research that has been widely investigated in recent decades (the monographs \cite{Berard1, BGM,  Chavel1, Henrot, Schoen-Yau} are among basic references on this subject). 
 In the present work we are interested in eigenvalues of ``weighted" energy functionals  with respect to ``weighted" $L^2$ inner products. Our aim is to investigate the interplay between the geometry of $(M,g)$ and the effect of  the weights. 
 
\smallskip
Therefore, let   $\rho $ and $\sigma $ be two positive continuous functions on $M$ and consider the Rayleigh quotient  
$$R_{(g,\rho,\sigma)}(u)=\frac{\int_M |\nabla u|^2 \sigma\, v_g}{\int_M  u^2 \rho \, v_g}.  $$
The corresponding eigenvalues are given by 
\begin{equation} \label{eq-1}
\mu_k^g(\rho,\sigma)=\inf_{E\in S_{k+1}} \sup_{u\in E\setminus \{ 0\} } R_{(g,\rho,\sigma)}(u).
\end{equation}
Under some regularity conditions on $\rho$ and $\sigma$, $\mu_k^g(\rho,\sigma)$ is the $k$-th eigenvalue of 
the  problem 
\begin{equation} \label{eq0}
-\mbox{div}(\sigma \nabla u)=\mu \rho u \qquad{\rm in }\ M
\end{equation}
with  Neumann conditions on the boundary if $\partial M\ne \emptyset$.
 Here $\nabla$ and $\mbox{div}$ are the gradient and the divergence associated with the Riemannian metric $g$. When there is no risk of confusion, we will simply write $\mu_k(\rho,\sigma)$ for $\mu_k^g(\rho,\sigma)$. 
 
\smallskip
Notice that the numbering of eigenvalues starts from zero. It is clear that the infimum of  $R_{(g,\rho,\sigma)}(u)$ is achieved by constant functions, hence  $\mu_0^g(\rho,\sigma)=0$ and  
\begin{equation} \label{intro0}
\mu_1^g(\rho,\sigma)= \inf_{\int_M u\rho v_g=0}  R_{(g,\rho,\sigma)}(u).
\end{equation}

\smallskip
One obviously has  $\mu_k^g(1,1)=\lambda_k(M,g)$. When $\sigma=1$, the eigenvalues $\mu_k(\rho,1)$ correspond to the situation where $M$ has a non necessarily constant mass density $\rho$ and describe, in dimension 2, the vibrations of  a  non-homogeneous membrane (see \cite{Laugesen, Henrot} and the references therein). The eigenvalues $\mu_k(1,\sigma)$ are those of the operator $\mbox{div}(\sigma \nabla u)$ associated with a conductivity  $\sigma$ on $M$ (see \cite[Chapter 10]{Henrot} and \cite{Astala}). In the case where $\rho=\sigma$, the eigenvalues $\mu_k(\rho,\rho)$ are those of the  Witten Laplacian  $L_\rho$ (see  \cite{CES} and the references therein).
Finally, when $\sigma$ and $\rho$ are related by  $\sigma=\rho^{\frac{n-2}n}$, the corresponding eigenvalues  $\mu_k^g(\rho,\rho^{\frac{n-2}n})$ are exactly those of the Laplacian associated with the conformal metric $\rho^{\frac2n}g$, that is $\mu_k^g(\rho,\rho^{\frac{n-2}n})=\lambda_k(M,\rho^{\frac2n}g)$.
 
 \smallskip
Our goal in this paper is to investigate the behavior of $\mu_k^g(\rho,\sigma)$, especially in the most significant cases mentioned above,
under normalizations that we will specify in the sequel, but  which essentially consist in the preservation of the total mass. The last case, corresponding to conformal changes of metrics, has been widely investigated in recent decades (see for instance \cite{ CE,  GNY, asma1, jammes3, Kokarev, Korevaar, Petrides2, Petrides1}) and most of the questions we will address in this paper are motivated by results established in the conformal setting.  
These questions can be listed as follows:  

\begin{enumerate}
\item  Can one redistribute the mass density $\rho$ (resp. the conductivity  $\sigma$) so that the corresponding eigenvalues become as small as desired?

\item  Can one redistribute $\rho$ and/or $\sigma$ so that the eigenvalues become as large as desired?

\item    If Question (1) (resp. (2)) is answered positively, what kind of constraint can one impose in order to get upper or lower bounds for the eigenvalues?

\item   If Question (1) (resp. (2)) is answered negatively, what are the geometric quantities that bound the eigenvalues?

\item   If the eigenvalues are bounded, what can one say about   their extremal values?

\item   Is it possible,  in some specific situations, to compute or to have sharp estimates for the first  positive eigenvalues?
\end{enumerate}

\smallskip

In a preliminary section we deal with some technical issues concerning the possibility of relaxing the conditions of regularity and positivity of the densities. In the process, we prove a 2-dimensional convergence result  (Theorem \ref{casesurfaces}) which completes a   theorem that Colin de Verdière had established in dimension $n\ge 3$  .  Question (1) is discussed at the beginning of   Section \ref{below} where we show that it is possible to fix one of the densities $\rho$ and $\sigma$ and vary the other one, among densities preserving the  total mass, in order to produce arbitrarily small eigenvalues (Theorem \ref{muinf}).  This leads us to get into Question (3) that we tackle by establishing the following  Cheeger-type inequality (Theorem \ref{Cheeger}):
$$
\mu_1(\rho,\sigma) \ge \frac{1}{4} h_{\sigma,\sigma}(M)h_{\rho,\sigma}(M)
$$
 where  $h_{\sigma,\sigma}(M)$ and $h_{\rho,\sigma}(M)$ are suitably defined isoperimetric constants, in the spirit of what is done in \cite{jammes1}. 

Whenever  a Cheeger-type inequality is proved, a natural question is to investigate a possible reverse inequality under some geometric restrictions (see \cite{Buser} and the introduction of  \cite{Milman} for a general presentation of this issue). It turns out that in the present situation, such a reverse inequality cannot be obtained without additional assumptions on the densities. Indeed, we prove that  on any given Riemannian manifold, there exists families 
 of  densities such that the associated Cheeger constants are as small as desired  while the corresponding eigenvalues are uniformly bounded from below (Theorem \ref{Buser-type}).

\smallskip
Questions (2) and (4) are addressed in Section \ref{above}.  A.  Savo and the authors
 have proved in \cite{CES}  that the first positive eigenvalue $\mu_1(\rho,\rho)$  of the Witten Laplacian   is not bounded above  as $\rho$ runs over densities of fixed total mass. In Proposition \ref{unbound} we prove that, given a  Riemannian metric $g_0$,  we can find a metric $g$, within the set of  metrics conformal to $g_0$ and of the same volume as $g_0$, and a density $\rho$, among densities of fixed total mass with respect to $g_0$, so that  $\mu_1^g(\rho,1)$ is as large as desired. The same also holds for $\mu_1^g(1,\sigma)$.

\smallskip
However, if instead of requiring that the total mass of the densities is fixed with respect to  $g_0$, we assume that it is fixed  with respect to $g$, then the situation changes completely. Indeed, 
Theorem \ref{bound-rho1} below gives   the following estimate when 
$M$ is  a  domain of  a complete Riemannian manifold $(\tilde M,g_0)$ whose Ricci curvature satisfies $Ric_{g_0}\ge -(n-1)$ (including the case $M=\tilde M$ if $\tilde M$ is compact): For every  metric $g$ conformal to $g_0$ and every  density 
$\rho $ on $M$ with $\int_{M}\rho v_g =\vert M\vert_g$, one has 
\begin{equation}\label{intro1}
 \mu_k^g(\rho,1)  \le  \frac 1{\vert M\vert_g^{\frac 2n}} \left(A_nk^{\frac 2n} + B_n \vert M\vert_{g_0}^{\frac 2n} \right),
\end{equation}
where $\vert \ .\ \vert_{g}$ and $\vert \ . \ \vert_{g_0}$ denote the Riemannian volumes with respect to $g$ and $g_0$, respectively, and $A_n$ and $B_n$ are two constants which depend only on the dimension $n$.

\smallskip
A direct consequence of this theorem  is the following inequality satisfied by any density 
$\rho $ on $(M,g)$ with $\int_{M}\rho v_g =\vert M\vert_g$:
\begin{equation}\label{intro2}
 \mu_k^g(\rho,1)  \le A_n\left(\frac k{\vert M\vert_{g}}\right)^\frac 2n+ B_n \mbox{ric}_0 
\end{equation}
where $\mbox{ric}_0$ is a positive number such that $Ric_{g}\ge-(n-1)\mbox{ric}_0\ g$  (see Corollary \ref{bound-rho}).

\smallskip
Regarding the eigenvalues $ \mu_k^g(1,\sigma)$, we are able to prove an estimate of the same type as \eqref{intro2}: 
For every positive density 
$\sigma $ on $(M,g)$ with $\int_{M}\sigma v_g =\vert M\vert_g$ one has (Theorem \ref{bound-sigma}) 
\begin{equation}\label{intro3}
\mu_k^g(1,\sigma) \le A_n \left(\frac k{\vert M\vert_{g}}\right)^\frac 2n + B_n \mbox{ric}_0,  
\end{equation}
where $A_n$ and $B_n$ are two constants which depend only on the dimension $n$. It is worth noting that although the estimates \eqref{intro2} and \eqref{intro3} are similar, their proofs are  of different nature. That is why we were not able to decide whether  a stronger   estimate such as \eqref{intro1} holds for $\mu_k^g(1,\sigma)$.

\smallskip
When $M$ is a bounded domain of a manifold $(\tilde M,\tilde g)$ of nonnegative Ricci curvature (e.g.  $\R^n$), the inequalities \eqref{intro2} and \eqref{intro3} give the following estimates that can be seen as extensions of Kröger's inequalitiy  \cite{Kro}:  $\mu_k^g(\rho,1)  \le A_n\left(\frac k{\vert M\vert_{g}}\right)^\frac 2n $  and $\mu_k^g(1,\sigma) \le A_n \left(\frac k{\vert M\vert_{g}}\right)^\frac 2n $, provided that $\int_{M}\rho v_g =\vert M\vert_g$ and $\int_{M}\sigma v_g =\vert M\vert_g$.  Notice that if we follow Kröger's approach, then we  get an upper bound of $\mu_k^g(\rho,1)$ which involves the gradient of $\rho$ and the integral of  $\frac 1\rho$ (see \cite{EHIS}).

\smallskip
According to  \eqref{intro2} and \eqref{intro3}, it is  natural to introduce the following  {\it  extremal eigenvalues} on a given Riemannian manifold $(M,g)$:
$$
\mu_k^*(M,g)=\sup_{\fint_M\rho \, v_g=1}\mu_k^g(\rho,1) \quad 
\mbox{and}
\quad
\mu_k^{**}(M,g)=\sup_{\fint_M\sigma v_g=1}\mu_k^g(1,\sigma)$$
In section \ref{extremal} we  investigate the qualitative properties of these quantities in the spirit of what we did in \cite{CE} for the {\it conformal spectrum}, thereby  providing some answers to Question (5).
For example, when $M$ is of dimension 2, we have the following lower estimate (see \cite[Corollary 1]{CE}):
$$\mu_k^*(M,g)\ge 8\pi \frac k{\vert M\vert_g}.$$
This means that, given any Riemannian surface $(M,g)$, endowed with the constant mass disribution $\rho=1$ (whose eigenvalues can be  very close to zero), it is always possible to redistribute the mass density $\rho$ so that the resulting eigenvalue $ \mu_k^g(\rho,1) $ is greater or equal to    $8\pi \frac k{\vert M\vert_g}$.

\smallskip
It turns out that this phenomenon is specific to the dimension 2. Indeed, we prove (Theorem \ref{inf_sup_rho}) that  on any 
 compact manifold $M$ of dimension $n\ge 3$, there exists  a 1-parameter family of Riemannian metrics $g_\varepsilon$ of volume $1$ such that
$$\mu_k^*(M,g_\varepsilon)\le C k\varepsilon^{\frac {n-2}n},$$
 where C is a constant which does not depend on $\varepsilon$.
This means that in dimension $n\ge 3$, there exist geometric situations that generate very small eigenvalues, regardless of how the mass density is distributed.  

\smallskip
Regarding the extremal eigenvalues $ \mu_k^{**}(M,g)$, a  similar result is proved (Theorem   \ref{inf_sup_sigma}) which is, moreover, also valid in dimension 2. 

\smallskip
Note however that  it is possible to construct examples of  Riemannian manifolds $(M,g)$ with very small eigenvalues (for the constant densities), for which $\mu_k^*(M,g)$ and $ \mu_k^{**}(M,g)$) are sufficiently large (see Proposition \ref{small_big}). 

\smallskip
The last part of the paper (Section \ref{examples}) is devoted to the study of the first extremal eigenvalues $\mu_1^*$ and $ \mu_1^{**}$. We give sharp estimates  of these quantities  for some standard examples or under  strong symmetry  assumptions.

\section{Preliminary results} \label{technicalfact}
This section is dedicated to some preliminary technical results. The reason is that in order to construct examples and counter-examples, it is often more convenient to use densities that are non smooth or which  vanish somewhere in the manifold.  The key arguments  used in the proof of these results rely on the method developed  by Colin de Verdière in \cite{colin}. 

\smallskip
Let $(M,g)$ be a compact Riemannian manifold, possibly with boundary.

\begin{prop}\label{smoothing} Let $\rho\in L^\infty(M)$ and $\sigma\in C^0(M)$ be two positive  densities on  $M$. For every $N\in\N^*$, 
there exist two sequences of smooth positive densities $\rho_p$ and $\sigma_p$ such that, $\forall k\le N$,
$$\mu_k(\rho_p,\sigma_p)\to \mu_k(\rho,\sigma)$$ 
as $p\to\infty$.
\end{prop}
\begin{proof} 
Using standard density results, let $\rho_p$ and $\sigma_p$ be  two sequences of smooth positive densities such that, $\rho_p$ converges to $\rho$ in $L^2(M)$ and $\sigma_p$ converges uniformly towards $\sigma$. Assume furthermore that $\frac 12 \inf \rho\le \rho_p\le 2\sup\rho$ almost everywhere and that (replacing $\sigma_p$ by $\sigma_p+\Vert \sigma_p-\sigma\Vert_\infty $ if necessary) $\sigma\le \sigma_p$ on $M$. Then the sequence of quadratic forms $q_p(u)=\int_M \vert\nabla u\vert^2 \sigma_p v_g$ together with the sequence of norms $\Vert u\Vert_p^2 =\int_M  u^2 \rho_p v_g$ satisfy the assumptions of Theorem I.8 of \cite{colin} which enables us to conclude.
\end{proof}

Let $M_0$ be a domain in  $M$ with $C^1$-boundary  and let $\rho$ be a positive bounded function on $M_0$. In order to state the next result, let us introduce the following quadratic form defined on $H^1(M_0)$:
 $$Q_0(u)=\int_{M_0}\vert\nabla u\vert^2 v_g + \int_{M\setminus M_0}\vert\nabla H(u)\vert^2 v_g $$
 where $H(u)$ is the harmonic extension of $u$ to $M\setminus M_0$, with Neumann condition on  $\partial M\setminus\partial M_0$ if $\partial M\setminus\partial M_0\ne\emptyset$ (i.e. $H(u)$ is harmonic on $M\setminus M_0$, coincides with $u$ on $\partial M_0\setminus\partial M$, and $\frac{\partial H(u)}{\partial \nu}=0$ on $\partial M\setminus\partial M_0$. The function $H(u)$ minimizes $\int_{M\setminus M_0}\vert\nabla v\vert^2 v_g $ among all functions $v$ on $M\setminus M_0$ which coincide with $u$ on $\partial M_0\setminus\partial M$). 
 We denote by $\gamma_k(M_0,\rho)$ the eigenvalues of this quadratic form with respect to the inner product of $L^2(M_0, \rho v_g)$ associated with $\rho$, that is,
 $$\gamma_k(M_0,\rho)= \inf_{E\in S^0_{k+1}} \sup_{u\in E\setminus \{ 0\} }\frac{\int_{M_0}\vert\nabla u\vert^2 v_g + \int_{M\setminus M_0}\vert\nabla H(u)\vert^2 v_g }{\int_{M_0}  u^2 \rho \, v_g}  $$
 where $S^0_k$ is the set of all $k$-dimensional vector subspaces of  $H^1(M_0)$.

\begin{prop}\label{zerorho} Let $M_0\subset M$ be a domain  with $C^1$-boundary and   let $\rho\in L^\infty(M_0)$ be a positive density    with $\essinf_{M_0}\rho >0$. Define, for every  $\varepsilon >0$, the density  $\rho_\varepsilon\in L^\infty(M)$ by 
 \[
\rho_\varepsilon (x)=
  \left\lbrace
  \begin{array}{ll}
   \rho(x)& \quad \text{if } x\in M_0\\
 \varepsilon & \quad \text{otherwise.}\\
  \end{array}
  \right.
 \]
Then, for every positive $k$, 
$\mu_k(\rho_\varepsilon,1)$ converges to $\gamma_k(M_0,\rho)$ as $\varepsilon\to 0$.
\end{prop}

\begin{proof}
The eigenvalues $\mu_k(\rho_\varepsilon,1)$  are those of the quadratic  form  $q(u)=\int_M \vert\nabla u\vert^2 v_g$, $u\in H^{1}(M)$, with respect to the inner product $\Vert u\Vert_\varepsilon^2 = \int_M u^2 \rho_\varepsilon v_g$. Set $M_\infty=M\setminus M_0$ and $\Gamma=\partial M_0\cap\partial M_\infty=\partial M_0\setminus\partial M$. We identify $H^1(M)$ with the space $\mathcal H_\varepsilon =\{v=(v_0,v_\infty)\in H^1(M_0)\times H^1(M_\infty):\ {v_\infty}_{\restriction_{\Gamma}}= \sqrt\varepsilon \,{v_0}_{\restriction_{\Gamma}}\}$ through  the  map  $\Psi_\varepsilon(u)=(u_{\restriction_{M_0}}, \sqrt\varepsilon\, u_{\restriction_{M_\infty}})$. We endow $\mathcal H_\varepsilon$ with the inner product given by $\Vert (v_0,v_\infty)\Vert_\rho^2 = \int_{M_0} v_0^2 \rho v_g+\int_{M_\infty} v_\infty^2 v_g$ and consider   the quadratic form $q_\varepsilon (v_0,v_\infty)=\int_{M_0} \vert\nabla v_0\vert^2 v_g+\frac 1\varepsilon\int_{M_\infty} \vert\nabla v_\infty\vert^2 v_g$, so that, for every $u\in H^1(M)$
 $$\Vert \Psi_\varepsilon (u)\Vert_\rho = \Vert u\Vert_\varepsilon\qquad \mbox{and }\qquad q_\varepsilon(\Psi_\varepsilon(u))=q(u).$$ 
 Therefore, the eigenvalues of the quadratic form $q:H^1(M)\to \R$  with respect to $\Vert \ \Vert_\varepsilon$ (i.e. $\mu_k^g(\rho_\varepsilon,1)$ ) coincide with those of $ q_\varepsilon:\mathcal H_\varepsilon\to \R$ with respect to $\Vert \ \Vert_\rho$.
 
 \smallskip
 The space $\mathcal H_\varepsilon$ decomposes into the direct sum $\mathcal H_\varepsilon=\mathcal K_0^\varepsilon\oplus \mathcal K_\infty^\varepsilon$ with 
  $\mathcal K_0^\varepsilon=\{(v_0,v_\infty)\in \mathcal H_\varepsilon \ :\ v_\infty \mbox{ is harmonic, and  } \frac{\partial  v_\infty}{\partial \nu}=0 \mbox{ on } \partial M\setminus\partial M_0 \mbox{ if }\partial M\setminus\partial M_0\ne\emptyset \}$, and $\mathcal K_\infty^\varepsilon=\{(v_0,v_\infty)\in \mathcal H_\varepsilon \ :\ v_0=0 \}$ (Indeed, $v=(v_0,v_\infty)=(v_0, \sqrt\varepsilon H(v_0))+ (0, v_\infty-\sqrt\varepsilon H(v_0))$). These two subspaces are $q_\varepsilon$-orthogonal and, denoting by $\lambda_1(M_\infty)$ the first eigenvalue of $M_\infty$ under Dirichlet boundary conditions on $\Gamma$  and Neumann boundary conditions on $\partial M_\infty\setminus\Gamma$,  we have, for every $v=(0,v_\infty)\in \mathcal K_\infty$, 
  $$q_\varepsilon(v)= \frac 1\varepsilon\int_{M_\infty} \vert\nabla v_\infty\vert^2 v_g \ge  \frac 1\varepsilon\lambda_1(M_\infty) \int_{M_\infty}  v_\infty^2 v_g =  \frac 1\varepsilon\lambda_1(M_\infty) \Vert v\Vert_\rho^2.$$
 Theorem I.7 of \cite{colin} then implies that, given any integer  $N>0$,  {\it the $N$ first eigenvalues  $\mu_k(\rho_\varepsilon,1)$  of $q_\varepsilon$  on $\mathcal H_\varepsilon $ are, for  sufficiently small $\varepsilon$, as close as desired to the eigenvalues of the restriction of $q_\varepsilon$ on $\mathcal K_0^\varepsilon$.}  
 
We still have to compare the eigenvalues  of   $q_\varepsilon$ on $\mathcal K_0^\varepsilon$, that we denote $\gamma_k(\varepsilon)$,  with the eigenvalues $\gamma_k(M_0,\rho)$ of $Q_0$ on $L^2(M_0,\rho v_g)$. For this, we make use of Theorem I.8 of \cite{colin}. Indeed, $ \mathcal K_0^\varepsilon$ can be identified to $H^1(M_0)$ through $\Psi_\varepsilon^0: u\in H^1(M_0) \mapsto (u, \sqrt\varepsilon H(u))\in \mathcal  K_0^\varepsilon$, which satisfies $\Vert\Psi_\varepsilon^0(u)\Vert_\varepsilon^2=  \int_{M_0}  u^2\rho v_g + \varepsilon\int_{M_\infty}  H(u)^2 v_g$ and $q_\varepsilon(\Psi_\varepsilon^0(u)) = Q_0(u) = \int_{M_0} \vert\nabla u\vert^2 v_g+\int_{M_\infty} \vert\nabla H(u)\vert^2 v_g$. Hence, we are led to compare, on $L^2(M_0)$,  the eigenvalues of the quadratic form $Q_0$ with respect to the following two scalar products:  $\Vert u\Vert_\rho^2 = \int_{M_0}  u^2\rho v_g$ and $\Vert u\Vert_\varepsilon^2=\int_{M_0}  u^2\rho v_g + \varepsilon\int_{M_\infty}  H(u)^2 v_g$. 

\smallskip
Now, since $H(u)$ is a harmonic extension of $u_{\restriction_\Gamma}$ to $M_\infty$, there exists a constant $C$, which does not depend on $\varepsilon$, such that
$\int_{M_\infty}  H(u)^2 v_g \le C \int_{\Gamma} u^2  v_{\bar g}$, where $\bar g$ is the metric induced on $\Gamma$ by $g$. Indeed, let $\eta$ be the solution in  $M_\infty$ of $\Delta \eta=-1$ with $\eta_{\restriction_\Gamma} =0$ and $\frac{\partial\eta}{\partial\nu} =0$ on $\partial M_\infty\setminus\Gamma$. Observe that we have $\eta\ge 0$ (maximum principle and Hopf Lemma) and, since   $\int_{M_\infty}  g(\nabla (\eta H(u)),\nabla H(u))   v_g =0$, $\int_{M_\infty}  g(\nabla \eta,\nabla H(u)^2)  v_g = -2\int_{M_\infty}  \eta \vert \nabla H(u)\vert ^2 v_g\le 0$.  Thus
$$\int_{M_\infty}  H(u)^2 v_g =- \int_{M_\infty}  H(u)^2 \Delta \eta\, v_g=\int_{M_\infty}  g(\nabla \eta,\nabla H(u)^2)  v_g + \int_{\Gamma} u^2 \frac{\partial\eta}{\partial\nu} v_{\bar g}   \le c  \int_{\Gamma} u^2  v_{\bar g} $$
where $c$ is an upper bound of $\frac{\partial\eta}{\partial\nu}$ on $\Gamma$. On the other hand, $  \int_{\Gamma} u^2 v_{\bar g} $ is controlled by $\Vert u\Vert^2_{H^{\frac 12}(\Gamma)}$ which in turn is controlled  (using boundary trace inequalities in $M_0$) by $\Vert u\Vert^2_{H^1(M_0)}$. Finally, there exists a constant $C$ (which depends on $\essinf_{M_0} \rho$ but not on $\varepsilon$) such that $\int_{M_\infty}  H(u)^2 v_g\le C (\int_{M_0}  u^2\rho v_g + \int_{M_0} \vert\nabla u\vert^2 v_g)$ and, then 
$$\Vert u\Vert_\varepsilon^2\le C(\Vert u\Vert_\rho^2 +Q_0(u)).$$
 Since $\Vert u\Vert_\varepsilon^2$ converges to $\Vert u\Vert_\rho^2 $ as $\varepsilon\to 0$, this implies, according to \cite[Theorem I.8]{colin} (see also \cite[Remark 2.14]{jammes2}),  that, for sufficiently small $\varepsilon$, { the $N$ first eigenvalues $\gamma_k(\varepsilon)$ of $Q_0$ with respect to $\Vert\ \Vert_\varepsilon$ are as close as desired to those, $\gamma_k(M_0,\rho)$,  of $Q_0$, with respect to $\Vert\ \Vert_\rho$. }
\end{proof}

Recall that in dimension 2, one has 
\begin{equation}\label{dim2}
\mu_k^g(\rho,1)=\lambda_k(M,\rho g).
\end{equation} 
An immediate consequence of Proposition  \ref{zerorho} is  the following result which completes Theorem III.1 of Colin de Verdière \cite{colin}. 

\begin{thm} \label{casesurfaces} Let $(M,g)$ be a compact Riemannian manifold  of dimension $n\ge2$ and let $M_0\subset M$ be a domain with boundary of class $C^1$. Let $g_{\varepsilon}$ be the a family of Riemannian metrics on $M$, with $g_{\varepsilon}=g$ on $M_0$ and $g_{\varepsilon}=\varepsilon g$ outside $M_0$. Let $k\ge1$. 
\begin{enumerate}
\item (Theorem III.1 of \cite{colin}) If $n\ge 3$, then  $\lambda_k(M,g_{\varepsilon})$ converges to $\lambda_k(M_0,g)$ as $\varepsilon \to 0$

\item If $n=2 $, then $\lambda_k(M,g_{\varepsilon})$ converges to $\gamma_k(M_0, 1)$ as $\varepsilon \to 0$.
\end{enumerate}
\end{thm}

From Proposition \ref{smoothing} and  Proposition \ref{zerorho} we can deduce the following two corollaries:

\begin{cor}\label{zero_rho} Let $\rho\in L^\infty(M_0)$ be a positive density on a  domain $M_0\subset M$ with boundary of class $C^1$. There exists a family of smooth positive densities $\rho_\varepsilon$ on $M$ such that $\int_M \rho_\varepsilon v_g$ tends to $\int_{M_0} \rho v_g$ and, for every $k\in\N^*$, 
$\mu_k(\rho_\varepsilon,1)$ converges to $\gamma_k(M_0,\rho)$ as  $\varepsilon\to 0$. 
\end{cor}

\begin{cor}\label{density_neumann} Let  $(M,g)$ be a compact manifold possibly with boundary and let $M_0\subset M$ be a domain with boundary of class $C^1$. For every integer $k>0$ and every $\varepsilon>0$, there exists a positive smooth density $\rho_\varepsilon$ on $M$ such that $\int_M \rho_\varepsilon v_g=\vert M\vert_g$ and  
$$\mu_k(\rho_\varepsilon,1)\ge \frac{\vert M_0\vert_g}{\vert M\vert_g}\lambda_k(M_0,g) -\varepsilon.$$
\begin{proof} Let $\rho$ be the density on $M_0$ defined by  $\rho= \frac{\vert M\vert_g}{\vert M_0\vert_g}$. We apply  Corollary \ref{zero_rho} taking into account that $\gamma_k(M_0, \rho)= \frac{\vert M_0\vert_g}{\vert M\vert_g}\gamma_k(M_0, 1) \ge \frac{\vert M_0\vert_g}{\vert M\vert_g}\lambda_k(M_0,g)$. 
\end{proof}

\end{cor}

\begin{rem}\label{cheeger2-d}
In dimension 2, it is clear from \eqref{dim2}
 that  the problem  of minimizing or maximizing $\mu_k^{g}(\rho,1)$ w.r.t. $\rho$ is equivalent to the problem  of minimizing or maximizing $\lambda_k(M,g)$ w.r.t.  conformal deformations of the metric $g$. 
In dimension $n\ge 3$, the two problems are completely different. To emphasize this difference, observe that, given a positive constant $c$, one has 
$$\inf _{\rho\le c} \mu_k^{g}(\rho,1)\ge \frac 1c \mu_k^{g}(1,1)=\frac 1c \lambda_k(M,g)>0$$
while 
$$\inf _{\rho\le c} \lambda_k(M,\rho g)=0.$$
Indeed, let $B_j$, $j\le k+1$ be a family of mutually disjoint   balls in $M$ and consider the density  $\rho_\varepsilon $ which is  equal to  $c$ on each $B_j$ and equal to $\varepsilon$ elsewhere. According to  \cite[Theorem III.1]{colin}, $\lambda_k(M,\rho_\varepsilon g)$  converges as $\varepsilon \to 0$ to the $(k+1)$-th  Neumann eigenvalue of the union of balls which is zero. 
\end{rem}


\section{Bounding the eigenvalues from below} \label{below}

\subsection{Non existence of ``density-free"  lower bounds.} Let $(M,g)$ be a compact Riemannian manifold of dimension $n\ge 2$, possibly with boundary, and denote by $[g]$ the set of all Riemannian metrics $g'$ on $M$ which are conformal to $g$ with $\vert M\vert_{g'}=\vert M  \vert_{g}$. It is well known that  $\lambda_k(M,g') $ can be as small as desired when $g'$ varies  within  $[g]$, i.e. $\inf_{g'\in [g]}\lambda_k(M,g) =0$ (Cheeger dumbbells). Since $\mu_k^{g}(\rho,\rho^{\frac{n-2}n})= \lambda_k(M, \rho^{\frac 2n}g)  $, this property is equivalent  to
\begin{equation}\label{cheegerdumbbell}
\inf_{\int_M \rho  v_{g}= \vert M\vert_g} \mu_k^{g}(\rho,\rho^{\frac{n-2}n})  =0.
\end{equation}
Let us denote by
$\mathcal R_0$ the set of positive smooth functions $\phi$ on $M$ satisfying $ \fint_M \phi v_{g} =1$, where $\fint_M \phi v_{g} =\frac1{\vert M\vert_{g}} \int_M \phi v_{g} $. 
The following theorem shows that  $\mu_k(\rho,\sigma) $ is not bounded below when  one of the densities $\rho,\sigma$ is  fixed and the second one is varying  within $\mathcal R_0$. We also deal with the case  $\sigma=\rho^p$, $p\ge 0$, which includes \eqref{cheegerdumbbell} and the case of the Witten Laplacian. 

\begin{thm}\label{muinf} For every positive integer $k$,  one has, $\forall p>0$

$$(i) \qquad  \inf_{ \rho\in \mathcal R_0 } \mu_k(\rho,1)  =0
$$
$$(ii) \qquad \inf_{\sigma  \in \mathcal R_0} \mu_k(1,\sigma)  =0
$$
$$(iii) \qquad  \inf_{\rho  \in \mathcal R_0}\mu_k(\rho,\rho^p)  =0.
$$

\end{thm}
 
\begin{proof}[Proof of Theorem \ref{muinf}]


\smallskip


{ $(i)$}:  In dimension 2 one has  $\mu_k(\rho,1) = \lambda_k(M,\rho g)$ and  the problem is equivalent to that of deforming conformally the metric $g$ into a metric $\rho g$ whose $k$-th eigenvalue is as small as desired. 
The existence of such a deformation is well known. 


\smallskip

Assume now that the dimension of $M$ is at least 3.
Let us choose a point  $x_0$  in $M$. The Riemannian volume of a geodesic ball $B(x,r)$  of radius $r$ in $M$ is asymptotically
equivalent, as $r\to 0$,  to $\omega_{n} r^{n}$,
where $\omega_{n}$ is the volume  of the unit ball  in the 
$n$-dimensional Euclidean space.   Therefore, there exist $\varepsilon_0 \in(0,1)$  sufficiently small and $N\in \N$ so that, for every $r < \frac {\varepsilon_0} N$ and every $x\in B(x_0,\varepsilon_0)$, 
\begin{equation}\label{eq1}
 \frac 12 \omega_{n} r^{n}\le |B(x,r)| \le 2\omega_{n} r^{n}.
\end{equation}

Fix a positive integer $k$ and let $\delta =\frac{n-2}{4}$ so that  $\delta  <\frac n2 -1$. One can choose $N\in \N$   sufficiently large so that, for every $\varepsilon < \frac {\varepsilon_0} N$,  the ball $B(x_0,\varepsilon)$ contains $k$ mutually disjoint balls of radius $2\varepsilon^{\frac n2 -\delta}$ (indeed, since $\frac n2 -\delta >1$, $2\varepsilon^{\frac n2 -\delta}$ is very small compared to $\varepsilon $ as the latter tends to zero).
We consider a smooth positive density $\rho_\varepsilon$ such that $\rho_\varepsilon = \frac 1{\varepsilon^n}$ inside $B(x_0,\varepsilon)$,  $\rho_\varepsilon=\varepsilon$ in $M\setminus B(x_0,2\varepsilon)$, and $\rho_\varepsilon \le \frac 1{\varepsilon^n}$ elsewhere. Thanks to \eqref{eq1}, one has
$$ \int_M \rho_\varepsilon  v_{g} \le \frac 1{\varepsilon^n}\vert B(x_0,2\varepsilon)\vert_g + \varepsilon \vert M\vert _g\le 2^{n+1}\omega_{n} + \varepsilon \vert M\vert _g .$$  
For simplicity, we set $\alpha= {\frac n2 -\delta}= \frac{n+2}{4}$ and denote by $x_1,\dots, x_k$ the centers of $k$ mutually disjoint balls of radius $2\varepsilon^{\alpha}$ contained in $B(x_0,\varepsilon)$. 

\smallskip
For each $i\le k$, we denote $f_i$ the function  which vanishes  outside  $B(x_i,2\varepsilon^{\alpha})$, equals $1$ in $B(x_i,\varepsilon^{\alpha})$, and $f_i(x)=2-\frac 1{\varepsilon^{\alpha}} d_{g}(x,x_i)$ for every $x$ in the annulus $B(x_i,2\varepsilon^{\alpha}) \setminus B(x_i,\varepsilon^{\alpha})$. The norm of the gradient of $f_i$ vanishes everywhere unless inside the annulus where we have $|\nabla f_i|= \frac 1{\varepsilon^{\alpha}}$. Thus, using \eqref{eq1},
$$\int_M f_i^2 \rho_\varepsilon v_g\ge \frac 1{\varepsilon^{n}} \int_{B(x_i,\varepsilon^{\alpha})} f_i^2 v_g = \frac{|B(x_i,\varepsilon^{\alpha})|}{\varepsilon^{n}}\ge\frac 12 \omega_n \varepsilon^{n(\alpha-1)}$$
and
$$\int_M |\nabla f_i|^2 v_g\le  \frac {|B(x_i,2\varepsilon^{\alpha})|}{\varepsilon^{2\alpha}}=2^{n+1} \omega_n \varepsilon^{\alpha(n-2)} .$$
Thus
$$R_{(g,\rho_\varepsilon,1)}(f_i)\le {2^{n+2}}\varepsilon^{n-2\alpha}={2^{n+2}}\varepsilon^{\frac{n-2}2} .$$
In conclusion, we have
$$\mu_k(\rho_\varepsilon,1)\le {2^{n+2}}\varepsilon^{\frac{n-2}2} $$
and
$$\mu_k\left(\frac{\rho_\varepsilon}{\fint_M \rho_\varepsilon  v_g},1\right)=\mu_k(\rho_\varepsilon,1)\fint_M \rho_\varepsilon  v_g\le {2^{n+2}}\left( \frac{2^{n+1}\omega_{n}} {\vert M\vert_g}\varepsilon^{\frac{n-2}2} + \varepsilon^{\frac{n}2} \right)  .$$
Letting $\varepsilon $ tends to zero we get the result. 

\smallskip

 \noindent $(ii)$:  The proof is similar to the previous one. For $\varepsilon$ sufficiently small, we may assume that there exist $k+1$ mutually disjoint balls $B(x_i,\varepsilon^2)$ inside a ball $B(x_0,\varepsilon)$ and  consider any function $\sigma_\varepsilon\in \mathcal R_0$ such that $\sigma_\varepsilon =\varepsilon ^5$ inside  $B(x_0,\varepsilon)$. For each $i\le k+1$, let $f_i$ be the   function which vanishes  outside  $B(x_i,2\varepsilon^{2})$, equals $1$ in $B(x_i,\varepsilon^{2})$, and $f_i(x)=2-\frac 1{\varepsilon^{2}} d_{g}(x,x_i)$  in  $B(x_i,2\varepsilon^{2}) \setminus B(x_i,\varepsilon^2)$. As before,
 $$\int_M f_i^2  v_g\ge  \int_{B(x_i,\varepsilon^{2})} f_i^2  dx \ge  |B(x_i,\varepsilon^{2})| \ge \frac{1}2\omega_n \varepsilon^{2n}$$
and
$$\int_M |\nabla f_i|^2\sigma_\varepsilon v_g\le \frac 1{\varepsilon^4 } \int_{B(x_i,2\varepsilon^{2})}\sigma_\varepsilon v_g \le \varepsilon |B(x_i,2\varepsilon^{2})|\le 2^{n+1} \omega_n \varepsilon^{2n+1}.$$
Thus
$$\mu_k(1,\sigma_\varepsilon )\le \max_{i\le k+1} \frac{\int_M |\nabla f_i|^2\sigma_\varepsilon v_g}{\int_M f_i^2  v_g}\le 2^{n+2} \varepsilon.$$

\smallskip
  \noindent $(iii)$:  
For sufficiently small $\varepsilon$, let  $B(x_i,4\varepsilon)$, $i\le k+1$, be  $k+1$ mutually disjoint balls of radius $4\varepsilon$ in $M$. As before, we can assume that, $\forall r\le 4\varepsilon$, $\frac 12 \omega_{n} r^{n}\le |B(x_i,r)| \le 2\omega_{n} r^{n}$. We define $\rho_\varepsilon$ to be equal to $\frac 1{\varepsilon^n}$ on each  of the balls $B(x_i,\varepsilon)$ and equal to $\varepsilon^n$ in the complement of $\cup_{i\le k} B(x_i,2\varepsilon)$. 
For every $i\le k+1$, the function $f_i$ defined to be equal to $1$ on $B(x_i,2\varepsilon)$ and $f_i(x)=2-\frac 1{2\varepsilon} d_{g}(x,x_i)$  in the annulus $B(x_i,4\varepsilon)\setminus B(x_i,2\varepsilon)$ and zero in the complement of  $B(x_i,4\varepsilon)$ satisfies 
$$\int_M f_i^2 \rho_\varepsilon v_g\ge  \int_{B(x_i,\varepsilon)} f_i^2 \rho_\varepsilon dx =  \frac 1{\varepsilon^n} |B(x_i,\varepsilon)| \ge \frac 12{ \omega_{n}}.$$
 On the other hand, $\forall p>0$,
$$\int_M |\nabla f_i|^2\rho_\varepsilon^p v_g =\varepsilon^{pn} \int_{B(x_i,4\varepsilon)\setminus B(x_j,2\varepsilon)} |\nabla f_i|^2 v_g = \varepsilon^{pn} \frac 1{4\varepsilon^{2}}|B(x_i,4\varepsilon)|  \le 2^{2n-1}\omega_n\varepsilon^{(p+1)n-2}.$$
 Thus
$$\mu_k(\rho_\varepsilon,\rho_\varepsilon^p )\le \max_{i\le k+1} \frac{\int_M |\nabla f_i|^2\sigma_\varepsilon v_g}{\int_M f_i^2  v_g}\le 2^{2n}\varepsilon^{(p+1)n-2}.$$
Regarding $\fint_M \rho_\varepsilon v_g$, it is clear that it is bounded both from above and from below by positive constants that are independent of $\varepsilon$, which enables us to conclude.
\end{proof}

\subsection{Cheeger-type inequality}

Theorem \ref{muinf} tells us that it is necessary to involve other quantities than the total mass in order to get lower bounds for the eigenvalues. Our next theorem gives a lower estimate which  is modeled on Cheeger's inequality, with  suitably defined isoperimetric constants, as was done by Jammes for Steklov eigenvalues \cite{jammes1}. 

\smallskip
Let $(M,g)$ be a compact Riemannian manifold, possibly with boundary. The classical Cheeger constant is defined by 
$$h(M)=\inf_{{\vert D\vert_g} \le\frac 12\vert M\vert_g }\frac{\vert\partial D\setminus\partial M\vert_g}{\vert D\vert_g}=\inf_{D\subset M} \frac{\vert\partial D\setminus\partial M \vert_g}{\min\{\vert D\vert_g, \vert M\vert_g - \vert D\vert_g\}}.$$
Given two  positive densities $\rho $ and $\sigma$ on $M$, we introduce the following Cheeger-type constant:
$$h_{\rho, \sigma}(M)=\inf_{\vert D\vert_\sigma\le\frac 12 \vert M\vert_\sigma }\frac{\vert\partial D\setminus\partial M\vert_\sigma}{\vert D\vert_\rho}$$
with $\vert D\vert_\sigma $ (resp. $\vert\partial D\setminus\partial M\vert_\sigma $) is  the $n$-volume of $D$ (resp. the $(n-1)$-volume of $\partial D\setminus\partial M$) with respect to the measure induced by $\sigma v_g$.
\begin{thm}\label{Cheeger} 
One has
$$\mu_1(\rho,\sigma)\ge \frac 14 h_{\sigma,\sigma}(M) h_{\rho, \sigma}(M).$$
\end{thm}

\begin{proof}
The proof follows the same general outline as the original proof by  Cheeger (see \cite{cheeger} and \cite{Buser2}). We give here a complete proof in the case where $M$ is a closed manifold.  The proof in the case $\partial M\ne\emptyset$ can be done analogously. Let $f$ be a  Morse function   such that the $\sigma$-volume of its positive nodal domain $\Omega_+(f)=\{f>0\} $ is less or equal to half the $\sigma$-volume of $M$. 
 For every $t\in (0,\sup f)$ excepting a finite number of values, the set $f^{-1}(t)$ is a regular hypersurface of $M$. We denote by $ v^t_g $ the measure induced on $f^{-1}(t)$ by $v_g$ and set $P_\sigma(t)= \int_{f^{-1}(t)}\sigma v^t_g$. The level sets of $f$ are denoted  $\Omega(t)=\{f>t\}$  and we set  $V_\sigma(t)=\int_{\Omega(t)}  \sigma v_g$ and $V_\rho (t)=\int_{\Omega(t)} \rho \, v_g$ . Using the co-area formula one gets
$$\int_{\Omega_+(f)} |\nabla f| \sigma v_g=  \int_0^{+\infty} P_\sigma(t)dt .$$
On the other hand, the same co-area formula gives 
$$V_\rho(t)= \int_t^{+\infty} ds\int_{f^{-1}(s)}  \frac \rho{|\nabla f| } v^s_g  .$$
 Thus
$$V_\rho'(t) = -\int_{f^{-1}(t)}  \frac \rho{|\nabla f| } v^t_g. $$
Now
$$\int_{\Omega_+(f)} f \rho \, v_g =  \int_0^{+\infty}dt \int_{f^{-1}(t)}  \frac{ f\rho}{|\nabla f| }   v^t_g  = \int_0^{+\infty} tdt\int_{f^{-1}(t)}  \frac \rho{|\nabla f| } v^t_g   = -\int_0^{+\infty} tV_\rho'(t) dt  $$
which gives after integration by parts
$$\int_{\Omega_+(f)} f\rho \, v_g= \int_0^{+\infty} V_\rho(t) dt .$$
Similarly, one has 
$$\int_{\Omega_+(f)} f \sigma v_g   = \int_0^{+\infty} V_\sigma(t) dt .$$
Since $P_\sigma(t)\ge h_{\sigma,\sigma}(M) V_\sigma(t)$ and $P_\sigma(t)\ge h_{\rho, \sigma}(M) V_\rho(t)$ we deduce 
$$\int_{\Omega_+(f)} |\nabla f|  \sigma v_g\ge \max \left\{ h_{\sigma,\sigma}(M) \int_{\Omega_+(f)} f \sigma v_g \ , \ h_{\rho, \sigma}(M) \int_{\Omega_+(f)} f \rho \, v_g \right\}.$$
Using Cauchy-Schwarz inequality we get
\begin{eqnarray}\label{eq 2}
\nonumber \int_{\Omega_+(f)} |\nabla f|^2 \sigma v_g &\ge& \frac 14\frac{\left( \int_{\Omega_+(f)} |\nabla f^2| \sigma v_g\right)^2 }{\int_{\Omega_+(f)}  f^2 \sigma v_g}
\ge\frac 14\frac{ h_{\sigma,\sigma}(M)h_{\rho, \sigma}(M) \int_{\Omega_+(f)}  f^2\sigma v_g \int_{\Omega_+(f)}  f^2 \rho \, v_g }{\int_{\Omega_+(f)} f^2\sigma v_g}\\ &=& \frac 14 h_{\sigma,\sigma}(M)h_{\rho, \sigma}(M)  \int_{\Omega_+(f)}  f^2 \rho \, v_g .
\end{eqnarray}

Now, let $m\in\R$ be such that $\vert \{f>m\}\vert_\sigma = \vert \{f<m\}\vert_\sigma =\frac 12 \vert M\vert_\sigma$ (such an $m$ is called a median of $f$ for $\sigma$). Applying \eqref{eq 2} to $f-m$ and $m-f$ we get 
$$
\int_{\{f>m\}} |\nabla f|^2 \sigma v_g \ge \frac 14 h_{\sigma,\sigma}(M)h_{\rho, \sigma}(M)  \int_{\{f>m\}}  (f-m)^2 \rho \, v_g $$
and
$$
\int_{\{f<m\}} |\nabla f|^2 \sigma v_g \ge \frac 14 h_{\sigma,\sigma}(M)h_{\rho, \sigma}(M)  \int_{\{f<m\}}  (f-m)^2 \rho \, v_g .$$
Summing up we obtain
$$
\int_{M} |\nabla f|^2 \sigma v_g \ge \frac 14 h_{\sigma,\sigma}(M)h_{\rho, \sigma}(M)  \int_{M}  (f-m)^2 \rho \, v_g. $$
Since $ \int_{M}  (f-m)^2 \rho \, v_g= \int_{M}  f^2 \rho \, v_g +m^2\vert M\vert_\rho -2m\int_{M}  f \rho \, v_g $, we deduce that, for every $f$ such that $\int_{M}  f \rho \, v_g =0$,
$$
\int_{M} |\nabla f|^2 \sigma v_g \ge \frac 14 h_{\sigma,\sigma}(M)h_{\rho, \sigma}(M)  \int_{M}  f^2 \rho \, v_g $$
which, thanks to \eqref{intro0}, implies the desired inequality.
\end{proof}

 
\begin{rem}
In dimension 2, Theorem \ref{Cheeger} can be restated as follows:
If $(M,g)$ is a compact Riemannian surface, then
\begin{eqnarray}\label{cheeger_impN}
\lambda_1(M,g)\ge \frac 14 \sup_{g'\in [g]}h_{g',g'}(M) h_{g,g'}(M)
\end{eqnarray}
where   $h_{g,g'}(M)=\inf_{\vert D\vert_{g'}\le\frac 12 \vert M\vert_{g'} }\frac{\vert\partial D\vert_{g'}}{\vert D\vert_{g}}$. Indeed, for any $g'\in [g]$ there exists a positive $\rho\in C^\infty(M)$ such that $g=\rho g'$. Thus, $\lambda_1(M,g)=\mu_1^{g'}(\rho,1)$ and \eqref{cheeger_impN} follows from  Theorem \ref{Cheeger}.  
This inequality can be seen as an improvement of  Cheeger's inequality since the right-hand side is obviously bounded below by  $h_{g,g}(M)^2$. Notice that in  \cite{Buser},  Buser  gives an example of a
 family of metrics on the 2-torus such that the Cheeger constant goes to zero while the first eigenvalue is bounded below. The advantage of  \eqref{cheeger_impN} is that its right hand side does not go to zero for Buser's example.

\end{rem}

A natural question is to investigate a possible reverse inequality of Buser's type  (see \cite{Buser, Milman}).  The following theorem 
provides a negative answer to this question.
 
 \begin{thm}\label{Buser-type} Let $(M,g)$ be a compact Riemannian manifold, possibly with boundary. 
 
 \smallskip
  \noindent (i) There exists a family of positive densities $\sigma_\varepsilon$, $\varepsilon>0$, on $M$ with $\fint_M \sigma_\varepsilon v_{g}=1$ and such that $h_{1,\sigma_\varepsilon}(M) h_{\sigma_\varepsilon,\sigma_\varepsilon}(M)$ goes  to zero with $\varepsilon$ while $\mu_1(1,\sigma_\varepsilon)$ stays bounded below by a constant $C$ which does not depend on $\varepsilon$.
  
   \smallskip
 \noindent (ii) There exists a family of positive densities $\rho_\varepsilon$, $\varepsilon>0$, on $M$ with $\fint_M \rho_\varepsilon v_{g}=1$ and such that $h_{\rho_\varepsilon,1}(M)$ goes to zero with $\varepsilon$ while $\mu_1(\rho_\varepsilon,1)$ stays bounded below by a constant $C$ which does not depend on $\varepsilon$.
 \end{thm} 

\begin{proof}

We start by proving the result for the unit ball $B^n\subset \R^n$ and then explain how to deduce it for  any compact Riemannian manifold.   
For every $r \in (0, 1)$  we denote by $B( r )$ the ball of radius $r$ centered at the origin and by $A_r$ the annulus $B^n\setminus B( r )$. In the sequel, whenever we integrate over a Euclidean set, the integration is implicitely made with respect  to the standard Lebesgue's measure. 

  \smallskip
\noindent \emph{Proof of (i)}: 
For every $\varepsilon \in (0,\frac 12)$  we  define a smooth nonincreasing radial density  $\sigma_\varepsilon$  on $B^n$ such that $\sigma_\varepsilon = \frac{1}{\varepsilon ^{1+a}}$, with $a\in(0,1)$ (e.g. $a=\frac 12$) 
inside  $B^n(\varepsilon)$ and $\sigma_\varepsilon = b_\varepsilon$  in $B^n\setminus B(2\varepsilon)$,
 where $b_\varepsilon$ is chosen so that 
$
\int_{B^n}\sigma_\varepsilon   = \omega_n,
$
the volume of $B^n$. 
We then have
$$
\int_{B(\varepsilon)} \sigma_\varepsilon =  \omega_n \varepsilon^{n-1-a}\qquad \mbox{ and }\qquad
\int_{A_{2\varepsilon}}\sigma_\varepsilon =\omega_n(1-2^n\varepsilon^n)b_\varepsilon.
$$
Since $
\int_{B^n} \sigma_\varepsilon   = \omega_n
$ and $b_\varepsilon\le \sigma_\varepsilon\le \varepsilon^{-1-a}$ on $B(2\varepsilon)\setminus B(\varepsilon)$,
we have
$$
 \omega_n \varepsilon^{n-1-a}+b_\varepsilon\omega_n(1-\varepsilon^n)\le\omega_n\le \omega_n 2^n \varepsilon^{n-1-a}+b_\varepsilon\omega_n(1-2^n\varepsilon^n),
$$
that is
\begin{equation} \label{controlb}
\frac{1-2^n\varepsilon^{n-1-a}}{1-2^n\varepsilon^n}\le b_\varepsilon\le \frac{1-\varepsilon^{n-1-a}}{1-\varepsilon^n}.
\end{equation}
Now, the Cheeger constant $h_{\sigma_\varepsilon,\sigma_\varepsilon}(B^n)$ satisfies 
$$h_{\sigma_\varepsilon,\sigma_\varepsilon}(B^n)\le \frac{\vert \partial B(2\varepsilon)\vert_{\sigma_\varepsilon}}{\vert B(2\varepsilon)\vert_{\sigma_\varepsilon} }\le \frac{\vert\partial B(2\varepsilon)\vert_{\sigma_\varepsilon}}{\vert B(\varepsilon)\vert_{\sigma_\varepsilon} } = \frac{nb_\varepsilon \omega_n(2\varepsilon)^{n-1}}{\omega_n\varepsilon^{n-1-a}}\le n2^{n-1}\varepsilon^{a}.$$
On the other hand, for $r_0 = \left(\frac14\right)^{\frac 1n}$ we have $\vert B( r_0 )\vert_{\sigma_\varepsilon} <\omega_n(\varepsilon^{n-1-a}+ \frac 14b_\varepsilon)  <\frac 12 \omega_n $ when $\varepsilon$ is sufficiently small, so that
$$h_{1,\sigma_\varepsilon}(B^n)\le \frac{\vert \partial B( r_0 )\vert_{\sigma_\varepsilon} }{\vert B( r_0 )\vert  } = \frac{n\omega_n r_0^{n-1}b_\varepsilon}{ \omega_n r_0^n}\le 4^{\frac 1n}n.$$
Hence,  the product $h_{1,\sigma_\varepsilon} (B^n)h_{\sigma_\varepsilon,\sigma_\varepsilon}(B^n)$ tends to zero as $\varepsilon\to 0$. 
Regarding the first positive eigenvalue $\mu_1( 1,\sigma_\varepsilon) $, if $f$ is a corresponding eigenfunction,  then $\int_{B^n} f  =0$ and 
 $$\mu_1( 1,\sigma_\varepsilon) = \frac{\int_{B^n} \vert\nabla f\vert^2\sigma_\varepsilon  }{\int_{B^n} f^2  } \ge b_\varepsilon \frac{\int_{B^n} \vert\nabla f\vert^2  }{\int_{B^n} f^2  }\ge b_\varepsilon\lambda_1(B^n, g_E)$$
with $b_\varepsilon\ge \frac12$ for sufficiently small $\varepsilon$ according to \eqref{controlb}.

\smallskip
Now, given a Riemannian manifold $(M,g)$, we fix a point $x_0$ and choose $\delta>0$ so that the geodesic ball $B(x_0,\delta)$ is $2$-quasi-isometric to the Euclidean ball of radius $\delta$. 
In the Riemannian manifold $(M,\frac1{\delta^2}g)$, the  ball $B(x_0,1)$ is 2-quasi-isometric to the Euclidean ball $B^n$. 
We define $\sigma_\varepsilon$ in  $B(x_0,1)$ as  the pull back of the  function $\sigma_\varepsilon$ constructed above, and extend it by   $b_\varepsilon $ in  $M\setminus B(x_0,1)$. Because of (\ref{controlb}), we easily see that $\fint_M \sigma_\varepsilon v_g$ stays bounded independently from $\varepsilon$. We can also  check that    $h_{1,\sigma_\varepsilon}(M)$ and $h_{\sigma_\varepsilon,\sigma_\varepsilon}(M)$ have the same behavior as before and that (since $\sigma_\varepsilon\ge b_\varepsilon\ge \frac 12$) the eigenvalue $\mu_1^{\delta^{-2}g}(1,\sigma_{\varepsilon})$ is bounded from below by $\frac{1}{2}\lambda_1(M,\delta^{-2}g)$ which is a positive constant $C$ independent  of $\varepsilon$. Thus, $\mu_1^{g}(1,\sigma_{\varepsilon})= \delta^2\mu_1^{\delta^{-2}g}(1,\sigma_{\varepsilon})\ge C\delta^2.$

  \medskip
\noindent
 \emph{Proof of (ii)}: As before we define  the density  $\rho_\varepsilon\in L^\infty(B^n)$, $\varepsilon\in(0,\frac 12)$, 
by 
\begin{equation}
\rho_\varepsilon = \left\{
 \begin{array}{lll}
\frac{1}{\varepsilon ^{1+a}} &  \   & \text{if} \ x\in B(\varepsilon)\\
    b_\varepsilon =\frac{1-\varepsilon^{n-1-a}}{1-\varepsilon^n}  &  \   &  \text{if}  \  x\in B^n\setminus B(\varepsilon)
\end{array}
\right.
\end{equation}
so that  $
\int_{B^n}\rho_\varepsilon  dx = \omega_n
$ and $b_\varepsilon <1$. 
The corresponding Cheeger constant satisfies
$$
h_{\rho_\varepsilon, 1} \le \frac{\vert \partial B(\varepsilon) \vert}{\vert B(\varepsilon)\vert_{\rho_\varepsilon}} = \frac{n\omega_n \varepsilon^{n-1}}{\omega_n  \varepsilon^{n-1-a}}= n\varepsilon^{a}.
$$
which goes to zero as $\varepsilon \to 0$.

\smallskip
To prove that the first positive Neumann eigenvalue $\mu_1(\rho_\varepsilon,1)$ is uniformly bounded below we will first prove that  the first Dirichlet  eigenvalue $\lambda_1(\rho_\varepsilon)$ satisfies
\begin{equation}\label{dirichlet}
\lambda_1(\rho_\varepsilon) \ge \frac{1}{4} \lambda^*
\end{equation}
where $\lambda^*$ is the first Dirichlet eigenvalue of the Laplacian on $B^n$. Indeed, let $f$ be a positive eigenfunction associated to $\lambda_1(\rho_\varepsilon)$. Such a function is necessarily a nonincreasing radial function and it satisfies (with $b_\varepsilon\le 1$)
\begin{equation}\label{rayl}
\lambda_1(\rho_\varepsilon)= \frac{\int_{B(\varepsilon)}\vert \nabla f\vert^2 + \int_{A_\varepsilon}\vert \nabla f\vert^2}{\int_{B(\varepsilon)}f^2\rho_\varepsilon + \int_{A_\varepsilon}f^2\rho_\varepsilon}\ge  \frac{\int_{B(\varepsilon)}\vert \nabla f\vert^2 + \int_{A_\varepsilon}\vert \nabla f\vert^2}{\varepsilon^{-1-a}\int_{B(\varepsilon)}f^2 + \int_{A_\varepsilon}f^2}
\end{equation}
 For convenience we assume that  $f(\varepsilon)=1$.

\smallskip
If we denote by $\nu(A_\varepsilon)$ the first eigenvalue of the mixed eigenvalue problem on the annulus $A_\varepsilon$, with Dirichlet  conditions on the outer boundary and Neumann conditions on the inner boundary, then it is well known that  $\nu(A_\varepsilon)$ converges to $\lambda^*$ as $\varepsilon \to 0$ (see\cite{Anne}). Thus, using the min-max, we will have for sufficiently small $\varepsilon$, 
\begin{equation}\label{annulus}
\int_{A_\varepsilon}\vert \nabla f\vert^2\ge \nu(A_\varepsilon) {  \int_{A_\varepsilon}f^2}\ge \frac12 \lambda^*{  \int_{A_\varepsilon}f^2}.
\end{equation}
On the other hand, since $f-1$ vanishes along $\partial B(\varepsilon)$, its Rayleigh quotient is bounded below by $\frac{1}{\varepsilon^2}\lambda^*$, the first Dirichlet eigenvalue of $B(\varepsilon)$. Thus
\begin{equation}\label{bal}
\int_{B(\varepsilon)} \vert \nabla f\vert^2\ge \frac{1}{\varepsilon^2}
\lambda^*
 \int_{B(\varepsilon)}(f-1)^2 \ge \frac{1}{\varepsilon^2}
\lambda^*\left( \int_{B(\varepsilon)} f^2 -2\int_{B(\varepsilon)}f\right)
\end{equation}
with
$$\int_{B(\varepsilon)}f\le \left(\omega_n\varepsilon^n \int_{B(\varepsilon)} f^2\right)^{\frac12} .$$
Thus, if $\omega_n\varepsilon^n \le \frac 1{16} \int_{B(\varepsilon)} f^2$, then \eqref{bal} yields 
$$\int_{B(\varepsilon)} \vert \nabla f\vert^2\ge \frac{1}{2\varepsilon^2}
\lambda^*
 \int_{B(\varepsilon)}f^2 > \frac{1}{2}
\lambda^* \varepsilon^{-1-a}\int_{B(\varepsilon)}f^2 $$
which, combined with \eqref{annulus} and \eqref{rayl}, implies \eqref{dirichlet}. 

Assume now that $\omega_n\varepsilon^n \ge \frac 1{16} \int_{B(\varepsilon)} f^2$ and let us prove the following:
 \begin{equation}\label{annulus2}
\int_{A_\varepsilon}\vert \nabla f\vert^2\ge  \left\{
 \begin{array}{lll}
  \frac{n(n-2)}{16\varepsilon^{1-a}} \ \varepsilon^{-1-a}\int_{B(\varepsilon)}f^2  &  \   & \text{if} \ n\ge 3\\
   \frac{1}{8 \varepsilon^{1-a}\ln(1/ \varepsilon)} \ \varepsilon^{-1-a}\int_{B(\varepsilon)}f^2 &  \   &  \text{if}  \  n=2
\end{array}
\right.
\end{equation}
which would imply for sufficiently small $\varepsilon$, 
 \begin{equation}\label{annulus3}
 \int_{A_\varepsilon}\vert \nabla f\vert^2\ge \frac 12 \lambda^*\varepsilon^{-1-a}\int_{B(\varepsilon)}f^2
 \end{equation}
enabling us to deduce  \eqref{dirichlet} from \eqref{rayl} and \eqref{annulus}.
Indeed, since $f(\varepsilon ) =1$ and $f(1)=0$, one has 
$
 \int_{\varepsilon}^{1} f' =-1.
$
Therefore, applying the Cauchy-Schwarz inequality to the product $f' = \left( f'r^{(n-1)/2}\right) r^{-(n-1)/2}$, we get
$$\frac 1{n\omega_n}\int_{A_\varepsilon}\vert \nabla f\vert^2 =\int_{\varepsilon}^{1} f'^2 r^{n-1}  \ge  \left(  \int_{\varepsilon}^{1} f' \right)^2 \left(\int_{\varepsilon}^{1}\frac{1}{r^{n-1}}\right)^{-1} \ge \frac 1{  \int_{\varepsilon}^{1}\frac{1}{r^{n-1}}}
$$
with 
\begin{equation}\label{calc1}
\int_{\varepsilon}^{1}\frac{1}{r^{n-1}} =  \left\{
 \begin{array}{lll}
  \frac1{n-2}\left( \frac 1{\varepsilon^{n-2}}-1 \right) <  \frac 1{n-2} \frac 1{\varepsilon^{n-2}} &  \   & \text{if} \ n\ge 3\\
  \ln({1} /\varepsilon)  &  \   &  \text{if}  \  n=2
\end{array}
\right.
\end{equation}
Therefore, 
 \begin{equation}\label{annulus1}
\int_{A_\varepsilon}\vert \nabla f\vert^2\ge  \left\{
 \begin{array}{lll}
  {n(n-2)\omega_n}\varepsilon^{n-2} &  \   & \text{if} \ n\ge 3\\
   \frac{2\pi}{\ln(1/ \varepsilon)}  &  \   &  \text{if}  \  n=2
\end{array}
\right.
\end{equation}
which gives \eqref{annulus2} since  $\omega_n\varepsilon^n \ge \frac 1{16} \int_{B(\varepsilon)} f^2$.

\smallskip
Let us check now that the first positive Neumann eigenvalue is also uniformly bounded from below.
Indeed, let $f$ be a Neumann eigenfunction with $\Delta f=-\mu_1(\rho_\varepsilon,1) \rho_\varepsilon f$. 
If $f$ is radial, then $\mu_1(\rho_\varepsilon,1)\ge \lambda_1(\rho_\varepsilon)\ge \frac{1}{4}\lambda^*$ (there exists $r_0<1$ with $f(r_0)=0$ so that  $f$ is a Dirichlet eigenfunction on the ball $B(r_0)$).
If  $f$ is not radial, then, up to averaging (or assuming that $f$ is orthogonal to radial functions), one can assume  w.l.o.g. that $\int_{\Sp^{n-1}( r )}   f d\theta=0$ for every $r<1$.
Thus, $\int_{\Sp^{n-1}( r )} \vert\nabla^{0} f\vert^2 d\theta\ge \frac{n-1}{r^2}\int_{\Sp^{n-1}( r )}   f^2 d\theta$, where $\nabla^0f$ is the tangential part of $\nabla f$. Hence,
$$\int_{B^n} \vert\nabla f\vert^2 =\int_0^1 r^{n-1} dr\int_{\Sp^{n-1}( r )} \vert\nabla f\vert^2 d\theta\ge (n-1) \int_0^1 r^{n-1} dr\int_{\Sp^{n-1}( r )}  \left( \frac{f}r\right)^2 d\theta$$
$$ \qquad \qquad\qquad \qquad \qquad \qquad \qquad \qquad = (n-1)\int_{B^n} \left( \frac{f}r\right)^2 \ge(n-1) \int_{B^n} f^2 \rho_\varepsilon$$
since $\rho_\varepsilon ( r )\le \frac 1{r^2}$ everywhere. Thus, in this case, 
$\mu_1(\rho_\varepsilon,1)\ge n-1$. Finally
$$\mu_1(\rho_\varepsilon,1)\ge \min (n-1,\frac{1}{4}\lambda^*).$$

As before, this construction can be implemented in any Riemannian manifold $(M,g)$, using a quasi-isometry argument, Proposition \ref{zerorho} and Corollary \ref{zero_rho}. 
\end{proof}

\smallskip
A relevant problem is to know if a Buser's type inequality can be obtained in this context under assumptions on the volume  of balls with respect to $\sigma$ and $\rho$.

\section{Bounding the eigenvalues from above} \label{above}

\subsection{Unboundedness of eigenvalues if only one parameter among $g, \rho,\sigma$  is fixed }

Let $(M,g_0)$ be a compact Riemannian manifold, possibly with boundary. Our first observation in this section is that the eigenvalues $\mu_k^g( \rho,\sigma)$ are not bounded from above when one quantity among $g\in [g_0], \rho \in \mathcal R_0,\sigma  \in \mathcal R_0$ is fixed and the two others are varying (here $\mathcal R_0=\{\phi\in C^\infty(M)\ : \phi>0 \mbox{ and}  \fint_M\phi \, v_{g_0}=1\}$). 

\smallskip

Let us first recall that  the authors and Savo  have proved in \cite{CES} that on any compact Riemannian manifold $(M,g_0)$ there exists a sequence of densities $\rho_j\in \mathcal R_0$ such that $\mu_1^{g_0}(\rho_j,\rho_j) $ tends to $+\infty$ with $j$. In particular, 
\begin{equation}\label{unbound1}
\sup_{\fint_M \rho  v_{g_0}= 1, \fint_M \sigma  v_{g_0}=1} \mu_1^{g_0}(\rho,\sigma) \ge \sup_{\fint_M \rho  v_{g_0}= 1}  \mu_1^{g_0}(\rho ,\rho ) =+\infty
\end{equation}
A natural subsequent question is: Can one construct examples of  $g\in [g_0]$ and $\rho\in \mathcal R_0$ (resp. $\sigma\in \mathcal R_0$) so that  $\mu_1^g(\rho,1) $  (resp. $\mu_1^g(1,\sigma)$) is as large as desired ?


\begin{prop}\label{unbound} Let $(M,g_0)$ be a compact Riemannian manifold, possibly with boundary. Then
\begin{equation}\label{unbound2}
\sup_{g\in [g_0],  \, \rho \in \mathcal R_0} \mu_1^g(\rho,1)  =+\infty
\end{equation}
and
\begin{equation}\label{unbound3}
\sup_{g\in [g_0], \,  \sigma \in \mathcal R_0} \mu_1^g(1,\sigma)  =+\infty.
\end{equation}

\end{prop}


\begin{proof} 
To prove \eqref{unbound2}, the idea is to deform both the metric and the density so that  $\rho_\varepsilon v_{g_\varepsilon}$ becomes everywhere small. Indeed,  let $V$ be an open set of $M$ with $\vert V\vert_{g_0} \ge \frac 1 {10}\vert M\vert_{g_0} $. For every $\varepsilon \in (0,1)$, we consider a continuous density  $\rho_\varepsilon $ such that $\rho_\varepsilon =\varepsilon$ on $V$, $\varepsilon \le \rho_\varepsilon \le 2 $ everywhere on $M$, and $\fint_M \rho_\varepsilon  v_{g_0}= 1$. Define $g_\varepsilon= \phi_\varepsilon^2 g_0$ with  
$$\phi_\varepsilon^n=\frac{\vert M\vert_{g_0} }{\int_M \rho_\varepsilon^{-1}  v_{g_0}} \ \frac 1 \rho_\varepsilon $$
so that $\vert M\vert_{g_\varepsilon} =\int_M \phi_\varepsilon^n  v_{g_0}= \vert M\vert_{g_0} $ (here $n$ denotes the dimension of $M$).
Now, we observe that 
$$\frac{1}{\varepsilon} \vert M\vert_{g_0}\ge {\int_M \rho_\varepsilon^{-1}  v_{g_0}}\ge {\int_V \rho_\varepsilon^{-1}  v_{g_0}} =\frac{1}{\varepsilon} \vert V\vert_{g_0}\ge \frac{1}{10 \varepsilon} \vert M\vert_{g_0}.$$
Thus,
$$\phi_\varepsilon^n\le \frac{10 \varepsilon}{\rho_\varepsilon}  $$
and, since $\rho_\varepsilon \le 2  $, 
$$\phi_\varepsilon^n\ge \frac{\varepsilon}{\rho_\varepsilon}  \ge \frac\varepsilon 2. $$
Now, for any smooth function $u$ on $M$ one has (with $ \frac\varepsilon 2\le\phi_\varepsilon^n\le \frac{10 \varepsilon}{\rho_\varepsilon}  $)
$$\frac{\int_M  \vert\nabla u\vert^2  v_{g_\varepsilon}}{\int_M u^2 \rho_\varepsilon  v_{g_\varepsilon}} = \frac{\int_M  \vert\nabla u\vert^2 \phi_\varepsilon^{n-2} v_{g_0}}{\int_M u^2 \rho_\varepsilon \phi_\varepsilon^n v_{g_0}} \ge  \frac 1 {2^{\frac{n-2}n}10\varepsilon^{\frac 2n}}   \frac{\int_M  \vert\nabla u\vert^2   v_{g_0}}{\int_M u^2  v_{g_0}}.$$
Therefore
$$\mu_1^{g_\varepsilon}(\rho_\varepsilon ,1)\ge \frac 1 {2^{\frac{n-2}n}10\varepsilon^{\frac 2n}}  \mu_1^{g_0}(1 , 1)$$
which tends to infinity as $\varepsilon$ goes to zero. 

\smallskip

To prove \eqref{unbound3} we first observe that, for any positive density $ \sigma$, one has, $\forall u\in C^2(M)$,
$$R_{(\sigma g_0, 1, \sigma)}(u) =R_{(g_0,   \sigma^{\frac n2}, \sigma^{\frac n2})}(u)$$
Thus, 
$$\mu_k^{\sigma g_0}( 1, \sigma)= \mu_k^{ g_0}(  \sigma^{\frac n2}, \sigma^{\frac n2}).$$
According to \cite{CES}, there exists on $M$ a sequence $ \sigma_j$ of positive densities such that $\int_M \sigma_j^{\frac n2}v_{g_0}=\vert M\vert_{g_0} $ and $  \mu_k^{ g_0}(\sigma_j^{\frac n2}, \sigma_j^{\frac n2})$  tends to infinity with $j$. We set $g_j=\sigma_j g_0 \in [g_0]$. Hölder inequality implies that 
$$\int_M \sigma_j v_{g_0}\le \left(\int_M \sigma_j^{\frac n2}v_{g_0}\right)^{\frac 2n} \vert M\vert_{g_0}^{1-\frac 2n} =\vert M\vert_{g_0}.$$
Setting $\sigma'_j=\frac {\sigma_j}{\fint_M \sigma_j v_{g_0}} \in \mathcal R_0$ we get
$$\mu_k^{ g_j}(1, \sigma'_j)=\frac 1{\fint_M \sigma_j v_{g_0}} \mu_k^{ \sigma_jg_0}( 1, \sigma_j)\ge \mu_k^{ \sigma_jg_0}( 1, \sigma_j)=\mu_k^{g_0}(  \sigma_j^{\frac n2}, \sigma_j^{\frac n2})$$
which proves that $\mu_k^{g_j} (1, \sigma'_j)$ tends to infinity with $j$.
\end{proof}

\subsection{Upper bounds for $\mu_k(\rho,1)$ and  $\mu_k(1,\sigma)$ }
Let $(M,g)$ be a compact Riemannian manifold of dimension $n\ge 2$, possibly with boundary. According to the result by 
Hassannezhad \cite{asma1} one has, when $M$ is a closed manifold,
\begin{equation}\label{asmaconf}
\lambda_k(M,g)\le \frac 1{\vert M\vert_g^{\frac 2n}} \left( A_nk^{\frac 2n}+B_n V([g])^{\frac 2n}\right)
\end{equation}
 where  $A_n$ and $B_n$ are two constants which only depend on $n$, and $V([g])$  is a conformally invariant geometric quantity defined as follows:
 $$ V([g])=\inf\{\vert M\vert_{g_0} \ : \ g_0 \mbox{ is conformal to } g \mbox{ and } Ric_{g_0}\ge -(n-1)g_0\}$$
 where $Ric_{g_0}$ is the Ricci curvature of $g_0$.
Now   for every positive $\rho$ such that $\fint_M \rho v_g =1$, we have $V([\rho^{\frac 2n}g])=V([g]) $, $\vert M\vert_{\rho^{\frac 2n}g}=\vert M\vert_g$ and $\lambda_k(M,\rho^{\frac 2n}g)=\mu_k^g(\rho, \rho^{\frac {n-2}n} )$. Hence, the inequality  \eqref{asmaconf} implies  that for every positive $\rho$ such that $\fint_M \rho v_g =1$, 
 \begin{equation}\label{asmaconfrho}
 \mu_k^g(\rho, \rho^{\frac {n-2}n} )\le \frac 1{\vert M\vert_g^{\frac 2n}} \left( A_nk^{\frac 2n}+B_n V([g])^{\frac 2n}\right)
.\end{equation}
This estimate is 
in contrast to what happens for  the Witten Laplacian where  we have 
$
\sup_{\fint_M \rho  v_{g}= 1} \mu_1^g(\rho,\rho)  =+\infty
$  (see \cite{CES}).  

\smallskip
Our aim in this section is to discuss the boundedness of $\mu_k^g(\rho, \sigma)$ in the two remaining important cases: $\mu_k^g(\rho,1)$ and  $\mu_k^g(1,\sigma)$. In \cite[Theorem 2.1]{CES} it has been shown that  the use of the GNY (Grigor'yan-Netrusov-Yau) method \cite{GNY} leads to the following estimate  
\begin{equation}\label{rhoboundGNY}
 \mu_k^g(\rho,1)  \fint_M \rho  v_{g}\le C( [g])\left( \frac k{\vert M\vert_g}\right)^{\frac2n}
 \end{equation}
where $C( [g])$ is a constant which only depends on the conformal class of the metric $g$.

\smallskip

This approach fails in the dual situation where $\sigma$ is varying while $\rho$ is fixed. Indeed, the GNY method leads to an upper bound of $\mu_k^g(1,\sigma) $ in terms of   the $L^{\frac{n-2}n}$-norm of $\sigma $ (instead of the $L^1$-norm). However, using the techniques developed by Colbois and Maerten in \cite{cm}, it is possible to obtain an inequality of the form 
\begin{equation}\label{sigmaboundcm}
 \mu_k^g(1,\sigma)  \le C( M,g)\left( \frac k{\vert M\vert_g}\right)^{\frac2n} \fint_M \sigma  v_{g}
 \end{equation}
where $C( M,g)$ is a geometric constant which does not depend on $\sigma$ (unlike \eqref{rhoboundGNY}, this method of proof does not allow to obtain a conformally invariant constant instead of $C(M,g)$). 

\smallskip

In what follows, we will establish inequalities of the type  \eqref{asmaconfrho} for $\mu_k(\rho,1)$ and  $\mu_k(1,\sigma)$.

\begin{thm}\label{bound-rho1} Let $M$ be a bounded open domain possibly with boundary of class $C^1$ of a  complete Riemannian manifold $(\tilde M, \tilde g_0)$ of dimension $n\ge 2$  (with $\tilde M=M$ if $\partial M=\emptyset$). Assume that  $Ric_{\tilde g_0}\ge -(n-1)\tilde g_0$ and let $g_0=\tilde g_0 \vert _M$.  For every metric $g $ conformal to $g_0$ and every  density 
$\rho $  with $\fint_{M}\rho v_g =1$, one has 
\begin{equation}\label{asmarho1}
 \mu_k^g(\rho,1)  \le   \frac 1{\vert M\vert_g^{\frac 2n}}\left(A_n  k^{\frac 2n}+ B_n\vert M\vert_{g_0}^{\frac 2n}\right)
\end{equation}
where $A_n$ and $B_n$ are two constants which depend only on  the dimension $n$.
\end{thm}

In the particular case where $(M,g)$ is a compact manifold without boundary, we can apply  Theorem \ref{bound-rho1} with $M =\tilde M$ and get immediately the following estimate which extends \eqref{asmaconf}:
\begin{equation}\label{asmarho2}
 \mu_k^g(\rho,1)  \le   \frac 1{\vert M\vert_g^{\frac 2n}}\left(A_n  k^{\frac 2n}+ B_nV([g])^{\frac 2n}\right).
\end{equation}
On the other hand, if $\tilde g$ is a metric on $\tilde M$ and if  $\mbox{ric}_0$ is  a  positive number such that $Ric_{\tilde g}\ge-(n-1)\mbox{ric}_0\ \tilde g$, then the metric $\tilde g_0=\mbox{ric}_0 \tilde g$ satisfies $Ric_{\tilde g_0}\ge-(n-1)\tilde g_0$ and   $\vert M\vert_{g_0}= \mbox{ric}_0^{n/2} \vert M\vert_g$, where $g=\tilde g \vert _M$ and $g_0=\tilde g_0 \vert _M$. Thus, we get

\begin{cor}\label{bound-rho} Let $M$ be a bounded open domain possibly with boundary of class $C^1$ of a  complete Riemannian manifold $(\tilde M, \tilde g)$ of dimension $n\ge 2$  (with $\tilde M=M$ if $\partial M=\emptyset$) and let $g=\tilde g \vert _M$.  For every  density 
$\rho $  with $\fint_{M}\rho v_g =1$, one has
\begin{equation}\label{asmarho}
 \mu_k^g(\rho,1)  \le  A_n\left(\frac k{\vert M\vert_{g}}\right)^\frac 2n+B_n \mbox{ric}_0
\end{equation}
where  $\mbox{ric}_0>0$ is  such that $Ric_{\tilde g}\ge-(n-1)\mbox{ric}_0\ \tilde g$. In particular, $\forall k\ge \vert M\vert_{g}  \mbox{ric}_0^{\frac n2}$,
\begin{equation}\label{asmarhoKroger}
 \mu_k^g(\rho,1)  \le  C_n\left(\frac k{\vert M\vert_{g}}\right)^\frac 2n
\end{equation}
with $C_n=A_n+B_n$.
\end{cor}
Inequalities \eqref{asmarho2} and \eqref{asmarho} are conceptually much stronger than \eqref{rhoboundGNY}, especially since they lead to a Kröger type inequality \eqref{asmarhoKroger} for every $k$ exceeding  an explicit geometric threshold, independent of $\rho$ (it is well known that if the Ricci curvature is not nonnegative, then an inequality like \eqref{asmarhoKroger} cannot hold for every $k$, see \cite[Remark 1.2(iii)]{cm}).  

 \begin{thm}\label{bound-sigma} Let $M$ be a bounded open domain possibly with boundary of class $C^1$ of a  complete Riemannian manifold $(\tilde M, \tilde g)$ of dimension $n\ge 2$  (with $\tilde M=M$ if $\partial M=\emptyset$) and let $g=\tilde g \vert _M$. 
 For every positive density 
$\sigma $ on $M$ with $\fint_{M}\sigma v_g =1$ one has 
\begin{equation}\label{asmasigma}
 \mu_k^g(1,\sigma)  \le  A_n \left(\frac k{\vert M\vert_{g}}\right)^\frac 2n +B_n \mbox{ric}_0 
\end{equation}
 where  $\mbox{ric}_0>0$ is  such that $Ric_{\tilde g}\ge-(n-1)\mbox{ric}_0\ \tilde g$ and where $A_n$ and $B_n$ are two constants which depend only on  $n$. In particular, $\forall k\ge \vert M\vert_{g}  \mbox{ric}_0^{\frac n2}$,
\begin{equation}\label{asmasigmaKroger}
 \mu_k^g(1,\sigma)  \le  C_n\left(\frac k{\vert M\vert_{g}}\right)^\frac 2n
\end{equation}
with $C_n=A_n+B_n$.

\end{thm}


\begin{proof} [Proof of Theorem \ref{bound-rho}]
 
We consider the metric measured space $(M,d_0,\nu)$ where $d_0$ is the restriction to $M$ of the Riemannian distance on $(\tilde M, \tilde g_0)$, and $\nu= \rho  v_{g}$. Since $Ric_{g_0}\ge -(n-1)g_0$, the  space $(M,d_0,\nu)$ satisfies a $(2,N;1)-$covering property for some fixed $N$ (see  \cite{asma1}). Therefore, we can apply Theorem 2.1 of \cite{asma1} and  find a family of $3(k+1)$ pairs of sets $(F_j,G_j)$ of $M$ with $F_j\subset G_j$, such that the $G_j$'s are mutually disjoint and $\nu(F_j)\ge \frac{\nu(M)}{c^2{(k+1)}}$, with $c=c(n)$ is a constant which depends only on $n$. Moreover, each pair $(F_j,G_j)$ satisfies one of the following properties: 

\smallskip
- $F_j$ is an annulus $A$ of the form $A=\{r<d_0(x,a)<R\}$, and $G_j=2A=\{\frac r2<d_0(x,a)<2R\}$, with outer radius $2R$  less than $1$,

\smallskip
- $F_j$ is an open set $V\subset M$ and $G_j=V^{r_0}=\{x\in M\ ; \ d_0(x,V)<r_0\}$, with $r_0= \frac{1}{1600}$.

\smallskip
Let us start with the case where $F_j$ is an annulus $A=A(a,r,R)=\{r<d_0(x,a)<R\}$ and $G_j=2A$.  To such an annulus we associate the function $u_A$ supported in $2A=\{\frac r2<d_0(x,a)<2R\}$ and such that 
\begin{equation}\label{testfunct}
u_A( x )=
  \left\lbrace
  \begin{array}{ll}
   \frac 2rd_0(x,a)-1 & \quad \text{if } \frac r2\le d_0(x,a)\le r\\
   1& \quad \text{if } x\in A\\
  2-\frac 1Rd_0(x,a) & \quad \text{if }  R\le d_0(x,a)\le 2R\\
  \end{array}
  \right.
 \end{equation}
 Since $u_A$ is supported in $2A$ we get, using Hölder's inequality and the conformal invariance of $\vert\nabla^{g} u_A\vert^n v_{g}$,  
$$\int_{M} \vert\nabla^{g}u_A\vert^2  v_{g} =\int_{2A } \vert\nabla^{g}u_A\vert^2  v_{g} \le \left(\int_{2A } \vert\nabla^{g} u_A\vert^n v_{g}\right)^{\frac 2n}{\left(\int_{2A} v_{g}\right)^{1-\frac2n}} $$
$$\qquad  \qquad \qquad \qquad \qquad \qquad \qquad = \left(\int_{2A} \vert\nabla^{g_0} u_A\vert^n v_{g_0}\right)^{\frac 2n} \vert 2A\vert_{g}^{1-\frac2n}.$$
Since
 \[
\vert\nabla^{g_0} u_A\vert \overset{a.e.}{=}
  \left\lbrace
  \begin{array}{ll}
   \frac 2 r & \quad \text{if } \frac r2\le d_0(x,a)\le r\\
   0& \quad \text{if } r\le d_0(x,a)\le R\\
  \frac 1R& \quad \text{if }  R\le d_0(x,a)\le 2R\\
  \end{array}
  \right.
 \]
 we get
 $$ \int_{2A} \vert\nabla^{g_0} u_A\vert^n v_{g_0} \le \left(\frac 2r \right)^n \vert B(a,r)\vert_{g_0} + \left(\frac 1R \right)^n  \vert B(a,2R)\vert_{g_0} \le 2^{n+1}\Gamma (g_0)$$
where 
$$\Gamma (g_0)=\sup_{x\in M, t\in(0,1)}\frac { \vert B(x,t)\vert_{g_0}}{t^n}$$
(here $B(x,t)$ stands for the ball of radius $t$ centered at $x$ in $(M,d_0)$). 
Notice that since $Ric_{\tilde g_0}\ge -(n-1)\tilde g_0$, the constant $\Gamma (g_0)$ is bounded above by a constant  that depends only on $n$
 (Bishop-Gromov inequality). Hence, 
$$ \int_{M} \vert\nabla^{g}u_A\vert^2  v_{g}\le C(n)  \vert2A \vert_{g}^{1-\frac2n} $$
where $C(n)\ge  2^{n+1}\Gamma (g_0)$.
 On the other hand, we have 
$$ \int_{M}  u_A^2\rho \, v_{g} \ge \int_{A }  \rho \, v_{g} =\nu(A)\ge  \frac{\nu(M)}{c^2{(k+1)}}.$$
Thus
$$R_{(g,\rho,1)}(u_A)=\frac{\int_{M} \vert\nabla^{g}u_A\vert^2  v_{g}}{\int_{M}  u_A^2\rho \, v_{g}}\le A_n \frac{  \vert 2A \vert_{g}^{1-\frac2n} }{ \nu(M) } (k+1)$$
for some constant $A_n$. 

Now, in the second situation, where $F_j$ is an open set $V$ and $G_j=V^{r_0}$, we introduce  the function $u_V$ defined to be equal to $1$ inside $V$, $0$ outside $V^{r_0}$ and proportional to the $d_0$-distance to the outer boundary in $V^{r_0}\setminus V$. We have, since $ u_V=1$ in $V$ and $ \vert \nabla^{g_0} u_V\vert$ is equal to $\frac1{r_0}$ almost everywhere in $V^{r_0}\setminus V$ and vanishes in $V$ and in $M\setminus V^{r_0}$,
$$\int_{M}  u_V^2\rho \, v_{g}\ge \int_{V}  \rho \, v_{g} =\nu(V)\ge \frac{\nu(M)}{c^2{(k+1)}}$$
and
$$
{\int_{M} \vert \nabla^{g} u_V\vert^2v_{g}}\le
{\left(\int_{V^{r_0}} \vert \nabla^{g} u_V\vert^nv_{g}\right)^{\frac 2n}\vert V^{r_0}\vert_{g}^{1-\frac 2n}}=  {\left(\int_{V^{r_0}} \vert \nabla^{g_0} u_V\vert^nv_{g_0}\right)^{\frac 2n}\vert V^{r_0}\vert_{g}^{1-\frac 2n}}$$
$$\le 
\frac{ \vert V^{r_0} \vert_{g_0}^{\frac 2n}\vert V^{r_0} \vert_{g}^{1-\frac 2n}}{{r_0}^2 }
$$
Thus
$$
R_{(g,\rho,1)}(u_V) \le 
B_n\frac{ { \vert V^{r_0} \vert_{g_0}^{\frac 2n}} \vert V^{r_0}  \vert_{g}^{1-\frac 2n}}{\nu(M)} (k+1)$$
where $B_n= \frac{c^2}{r_0^2}$ is a constant which depends only on $n$. 

\smallskip
In conclusion, to each pair $(F_j,G_j)$ we associate a test function $u_j$ supported in $G_j$ and  satisfying 
either $R_{(g,\rho,1)}(u_j)  \le A_n \frac{  \vert G_j \vert_{g}^{1-\frac2n} }{ \nu(M) } (k+1)$ or $R_{(g,\rho,1)}(u_j) \le B_n\frac{ { \vert G_j \vert_{g_0}^{\frac 2n}} \vert G_j \vert_{g}^{1-\frac 2n}}{\nu(M)} (k+1)$, that is
$$R_{(g,\rho,1)}(u_j)  \le A_n \frac{  \vert G_j \vert_{g}^{1-\frac2n} }{ \nu(M) } (k+1) + B_n\frac{ { \vert G_j \vert_{g_0}^{\frac 2n}} \vert G_j \vert_{g}^{1-\frac 2n}}{\nu(M)} (k+1).$$
Now,  observe that since  $\sum_{j\le 3(k+1)}\vert G_j  \vert_{g_0}\le \vert M\vert_{g_0}$ and $\sum_{j\le 3(k+1)}\vert G_j\vert_{g}\le \vert M\vert_{g}$, there exist at least $k+1$ sets among   $G_1,\dots, G_{3(k+1)}$ satisfying both  $\vert G_j\vert_{g_0}\le \frac{\vert M\vert_{g_0}}{k+1}$ and $\vert G_j\vert_{g}\le \frac{\vert M\vert_{g}}{k+1}$. 
This leads to a subspace of $k+1$ disjointly supported functions $u_j$  whose Rayleigh quotients are such that 
$$R_{(g,\rho,1)}(u_j)  \le A_n \frac{  \vert G_j \vert_{g}^{1-\frac2n} }{ \nu(M) } (k+1) + B_n\frac{ { \vert G_j \vert_{g_0}^{\frac 2n}} \vert G_j \vert_{g}^{1-\frac 2n}}{\nu(M)} (k+1)$$
$$ \qquad \quad \le A_n \frac{ {\vert M\vert_{g}}^{1-\frac2n} }{ \nu(M) } (k+1)^{\frac2n} + B_n \frac{ \vert M\vert_{g_0}^{\frac 2n}}{ \nu(M) }\vert M\vert_{g}^{1-\frac 2n} $$
with $\nu(M)=\int_M \rho v_{g}=\vert M\vert_g$. The desired inequality then immediately follows thanks to \eqref{eq-1}.  
\end{proof}

\begin{proof} [Proof of Theorem \ref{bound-sigma}]
 
 First, observe that it suffices to prove the theorem when $ \mbox{ric}_0 =1$ (i.e. $Ric_{\tilde g}\ge -(n-1)\tilde g$). Indeed,   the Riemannian metric $\tilde g_0= {\mbox{ric}_0} \tilde g$ satisfies $Ric_{\tilde g_0}\ge -(n-1)\tilde g_0$ and 
$\vert M\vert_{g_0}= (\mbox{ric}_0)^{n/2} \vert M\vert_g$, with $g_0=\tilde g_0\vert M$. Hence, the inequality
$$
\mu_k^{g_0}(1,\sigma)  \le A_n \left(\frac k{\vert M \vert_{g_0}}\right)^\frac 2n +B_n 
$$
implies
$$
\mu_k^g(1,\sigma) = \mbox{ric}_0 \mu_k^{g_0}(1,\sigma)\le \mbox{ric}_0\left( A_n  \left(\frac k{\vert M\vert_{g_0}}\right)^\frac 2n+B_n \right)= A_n \left(\frac k{\vert M\vert_{g}}\right)^\frac 2n + B_n \mbox{ric}_0  .
$$
Therefore, assume that  $ \mbox{ric}_0 =1$ and consider the metric measured space $(M,d,v_{g})$ where $d$ is the restriction to $M$ of the Riemannian distance of $(\tilde M, \tilde g)$. 
The proof relies on the method developed by Colbois and Maerten \cite{cm} as presented in Lemma 2.1 of \cite{ceg1}. 
Applying Bishop-Gromov Theorem, we deduce that there exist two constants, $C_n$ and $N_n$, depending only on $n$, such that, $\forall x\in M$ and $\forall r\le 1$,  
\begin{itemize}
 \item 
  $\vert B(x,r)\vert_{g}\le C_n r^n$
\item  
  $B(x,4r)$ can be covered by $N_n$ balls of radius $r$
 
\end{itemize}
where $B(x,r)$ stands for the ball in $M$ of radius $r$ with respect to the distance $d$. 

Let $k_0$ be the smallest integer such that  $2(k_0+1)> \frac{ \vert M \vert_{g}}{4 C_n N_n^2}$. For every $k\ge k_0$ we  define  $r_k$  by 
$$ r_k^n = \frac{\vert M \vert_{g}}{8C_nN_n^2 (k+1)}\le  1$$
which means that, $\forall x\in M$,   $$
  \vert B(x,r_k) \vert_{g} \le C_n r_k^n\le \frac{\vert M\vert_{g}}{8N_n^2 (k+1)}.
  $$
Thus, we can apply Lemma 2.1 of \cite{ceg1} and deduce the existence of 
$2(k+1)$ measurable subsets $A_1,\dots,A_{2(k+1)}$ of $M$
  such that, $\forall i\le 2(k+1)$, $\vert A_i \vert_{g}\ge \frac{\vert M\vert_{g}}{4N_n(k+1)}$
  and, for
  $i\not =j$, $d(A_i,A_j) \ge 3r_k$.
 To each set $A_j$ we associate the function $f_j$ supported in $A_j^{r_k}=\{x\in M\, :\, d(x,A_j)<r_k\}$  and defined to be equal to $1$ inside $A_j$  and proportional to the distance to the outer boundary in $A_j^{r_k}\setminus A_j$. The length of the gradient $ \vert \nabla^{g} f_j\vert$ is then equal to $\frac1{r_k}$ almost everywhere in $A_j^{r_k}\setminus A_j$ and vanishes elsewhere, so that we get
$$
R_{(g,1,\sigma)}(f_j) = \frac{\int_{A_j^{r_k}} \vert \nabla^{g} f_j\vert^2\sigma v_{g}}{\int_{A_j^{r_k}}  f_j^2 v_{g}}\le\frac{\frac 1 {{r_k}^2 } \int_{A_j^{r_k}} \sigma v_{g}}{ \vert A_j\vert_{g}}\le \frac {4N_n} {{r_k}^2 }  \frac{ \int_{A_j^{r_k}} \sigma v_{g} }{ \vert M\vert_{g}} (k+1)
$$
which gives, after replacing $r_k$ by its explicit value, 
$$
R_{(g,1,\sigma)}(f_j) \le A_n \frac{ \int_{A_j^{r_k}} \sigma v_{g} }{ \vert M \vert_{g}^{1+\frac2n}} (k+1)^{1+\frac2n}.
$$
for some constant $A_n$.
Now, since  $\sum_{j\le 2(k+1)} \int_{A_j^{r_k}} \sigma v_{g}\le  \int_{M} \sigma v_{g}$, there exist at least $k+1$ sets among the $A_j$'s such that  $ \int_{A_j^{r_k}} \sigma v_{g}\le \frac{\int_{M} \sigma v_{g}}{k+1}$. 
This leads to a subspace of $k+1$-disjointly supported functions $f_j$  whose Rayleigh quotients are such that 
$$
R_{(g,1,\sigma)}(f_j) \le A_n \frac{ \int_{M} \sigma v_{g} }{ \vert M\vert_{g}^{1+\frac2n}} (k+1)^{\frac2n}.
$$
Consequently, we have thanks to \eqref{eq-1}, for all $k\ge k_0$, 
$$
\mu_k^g(1,\sigma) \le A_n \frac{ \int_{M} \sigma v_{g} }{ \vert M\vert_{g}^{1+\frac2n}} (k+1)^{\frac2n} = A_n\left( \frac{  k+1 }{ \vert M\vert_{g}}\right)^{\frac2n}  
$$
since we have assumed that $ \int_{M} \sigma v_{g} =\vert M\vert_{g}$.
On the other hand, for every $k\le k_0$, one obviously has (since $k_0+1\le \frac{ \vert M\vert_{g}}{4 C_n N_n^2}$)
$$\mu_k^g(1,\sigma) \le \mu_{k_0}^g(1,\sigma) \le  A_n\left( \frac{  k_0+1 }{ \vert M\vert_{g}}\right)^{\frac2n}  
\le A_n \left(\frac{ 1}{4 C_n N_n^2}\right)^{\frac2n}.$$ 
Denoting by $B_n$ the latter constant we obtain, for every $k\ge 0$,
$$\mu_k^g(1,\sigma)\le A_n\left( \frac{  k}{ \vert M\vert_{g}}\right)^{\frac2n}+ B_n  .$$
\end{proof}

\section{Extremal eigenvalues } \label{extremal}

Let $(M,g)$ be a compact Riemannian manifold of dimension $n\ge2$, possibly with boundary. In \cite{CE},  we introduced the following conformally invariant quantities that we named ``conformal eigenvalues": For every $k\in\N$, $\lambda_k^c(M,[g])$ is defined as the supremum of $\lambda_k(M,g')$ when $g'$ runs over all metrics of unit volume which are  conformal to $g$ (or, equivalently, $\lambda_k^c(M,[g]) =\sup \lambda_k(M,g')\vert M\vert_{g'}^{\frac 2n}$ when $g'$ runs over all metrics conformal to $g$). Thus, we can write 
$$\lambda_k^c(M,[g])=\sup_{\int_M\rho \, v_g=1 }\lambda_k(M,{\rho^{\frac{2}n} g})=\sup_{\int_M\rho \, v_g=1}\mu_k^g(\rho,\rho^{\frac{n-2}n}).$$
We investigated in  \cite{CE} some of the properties of the conformal eigenvalues such as 
the existence of a universal lower bound, and proved that 
\begin{equation}\label{confspec1}
\lambda_k^c(M,[g])\ge\lambda_k^c(\Sp^n,[g_s])\ge n\alpha_n^{\frac2n}k^{\frac2n}
\end{equation}
where $\alpha_n=(n+1)\omega_{n+1}$ is the volume of the standard sphere. Moreover, we proved that the gap between two consecutive conformal eigenvalues satisfies the following  estimate:
\begin{equation}\label{confspec2}
\lambda_{k+1}^c(M,[g])^{\frac n2}-\lambda_k^c(M,[g])^{\frac n2}\ge n^{\frac n2}\alpha_n.
\end{equation}
Actually, these properties were established in the context of closed manifolds. However, they remain valid in the context of  bounded domains, under Neumann boundary conditions,  without the need to change anything to the proofs.
In this regard, we can point out the following  curious phenomenon that all bounded Euclidean domains have the same conformal spectrum.

\begin{prop}\label{conformaleucdomain} For every  bounded domain $\Omega\subset \R^n$ with $C^1$-boundary  one has
$$\lambda_k^c(\Omega,[g_E])=\lambda_k^c(B^n,[g_{E}])$$
where $g_E$ is the Euclidean metric.

\end{prop}
For $k=1$ we have $\lambda_1^c(\Omega,[g_E])=n\alpha_n^{\frac2n}$ (see Corollary \ref{mu1sup} below). 
\begin{proof}
Let us first observe that  if $\Omega$ is a proper subset of $\Omega'$, then $\lambda_k^c(\Omega,[g_E])\le\lambda_k^c(\Omega',[g_E])$. Indeed, 
given a metric $g=fg_E$  conformal to $g_E$ on $\Omega$, we extend it to $\Omega'$ in  a metric  $g'$ conformal to $g_E$.  For every $\varepsilon >0$, we multiply $g'$ by the function $f_\varepsilon$ which is equal to $1$ on $\Omega$ and equal to $\varepsilon$ on $\Omega'\setminus\Omega$ and apply Theorem \ref{casesurfaces}.   In dimension $n\ge3$, this theorem tells us that  $\lambda_k(\Omega',f_\varepsilon g')$ converges to $\lambda_k(\Omega,g)$. Since the volume of $(\Omega',f_\varepsilon g')$ converges to the volume of $(\Omega,g)$, we deduce that $\lambda_k(\Omega,g)\vert \Omega\vert_g^{2/n}\le\lambda_k^c(\Omega',[g_E])$.
 In dimension 2, we obtain that $\lambda_k(\Omega',f_\varepsilon g')$ converges to the $k$-th eigenvalue of the quadratic form $\int_\Omega\vert\nabla u\vert^2v_g+\int_{\Omega'\setminus\Omega}\vert\nabla H(u)\vert^2v_g$. This quadratic form is clearly larger than the Dirichlet energy $\int_\Omega\vert\nabla u\vert^2v_g$ on $\Omega$ so that its $k$-th eigenvalue is bounded below by $\lambda_k(\Omega,g)$.   Again, this implies that $\lambda_k(\Omega,g)\le\lambda_k^c(\Omega',[g_E]). $

\smallskip
Now,
since $\Omega$ is open and bounded, there exist two positive radii $r_1$ and $r_2$ so that 
$$B^n(r_1)\subset \Omega\subset B^n(r_2)$$
where $B^n(r_1)$ and $B^n(r_2)$ are two concentric Euclidean balls.  Using the observation above we get 
$$\lambda_{k}^c(B^n(r_1),[g_E])  \le\lambda_{k}^c(\Omega,[g_E]) \le\lambda_{k}^c(B^n(r_2),[g_E]) . $$
 Since the balls $B^n(r_1)$ and $ B^n(r_2)$ are homothetic to the unit ball $B^n$, one necessarily has $\lambda_{k}^c(B^n(r_1),[g_E])= \lambda_{k}^c(B^n(r_2),[g_E])=\lambda_{k}^c(B^n,[g_E]) $ which enables us to conclude.
\end{proof}

As a consequence of the upper bounds given in the previous section, it is natural to introduce the following extremal eigenvalues:

$$\mu_k^*(M,g)=\sup_{\fint_M\rho \, v_g=1}\mu_k^g(\rho,1) = \sup_{\rho }\mu_k^g(\rho,1)\fint_M\rho \, v_g$$
$$\mu_k^{**}(M,g)=\sup_{\fint_M\sigma v_g=1}\mu_k^g(1,\sigma)=\sup_{\sigma }\frac {\mu_k^g(1,\sigma)}{\fint_M\sigma v_g}$$

A natural question is whether properties such as \eqref{confspec1} and \eqref{confspec2} may occur for $\mu_k^*(M,g)$ and $\mu_k^{**}(M,g)$. 
 Observe that these quantities are not invariant under metric scaling since 
 $$\mu_k^*(M,r^2g)=r^{-2}\mu_k^*(M,g)\quad  and \quad \mu_k^{**}(M,r^2g)=r^{-2}\mu_k^{**}(M,g).$$ 
Hence, we will assume  that the volume of the manifold is fixed.


\smallskip
In the particular case of manifolds   $(M,g)$ of dimension 2  one has for every  $\rho$, $\mu_k^g(\rho, 1)=\lambda_k(M,\rho g)$. Thus, 
 \begin{equation}\label{conform_rho}
\mu_k^*(M,g) = \frac{\lambda_k^c(M,g)}{\vert M\vert_g}
 \end{equation}
and we deduce from \eqref{confspec1} and \eqref{confspec2} that any 2-dimensional Riemannian manifold $(M,g)$  satisfies
$$\mu_k^*(M,g) \ge \frac {8\pi k}{\vert M\vert_g}$$
and
$$\mu_{k+1}^*(M,g)-\mu_k^*(M,g) \ge   \frac {8\pi }{\vert M\vert_g}.  $$

\smallskip
The following theorem shows that the 2-dimensional case is in fact exceptional. Indeed, it turns out that any compact manifold of dimension $n\ge 3$ can be deformed in such a way that $\mu_k^*(M,g)$ becomes as small as desired.

\begin{thm}\label{inf_sup_rho} Let $M$ be a compact manifold of dimension $n\ge 3$.  There exists on $M$ a one-parameter family of metrics $g_\varepsilon$, $\varepsilon>0$, of volume 1  such that 
$$\mu_k^*(M,g_\varepsilon)\le C k\, \varepsilon^{\frac {n-2}n},$$
 where C is a constant which does not depend on $\varepsilon$ or $k$.

\end{thm}

Similarly, we have the following result for the supremum with respect to $\sigma$.

\begin{thm}\label{inf_sup_sigma} Let $M$ be a compact manifold of dimension $n\ge 2$.  There exists on $M$ a one-parameter family of metrics $g_\varepsilon$, $\varepsilon>0$, of volume 1  such that 
$$\mu_k^{**}(M,g_\varepsilon)\le C k^2 \varepsilon^{2\frac {n-1}n}$$
where $C$ is a constant which depends only on $n$.

\end{thm}

The proofs of these  theorems rely on the  construction below. It is worth noticing that the one-parameter family of metrics $g_\varepsilon$ we will exhibit can be chosen within a fixed conformal class.  
Actually, we start with a Riemannian metric $g_0$ on $M$ that we conformally deform in the neighborhood of a point. 

\smallskip
\noindent
\textit{The construction.} We start with a metric $g_0$ on $M$ and choose a sufficiently small open set $V\subset M$ so that $g_0$ is 2-quasi-isometric to a flat metric in $V$. Since the eigenvalues corresponding to two quasi-isometric metrics are ``comparable",  we can assume w.l.o.g. that the metric $g_0$ is flat inside $V$. Therefore, there exists a positive $\delta$ so that $V$ contains a flat (Euclidean) ball of radius $\delta$.  After a possible dilation, we can assume that $\delta =1$. We deform this unit Euclidean ball   into a long capped cylinder (i.e. an Euclidean cylinder of radius $\delta$ closed by a spherical cap). This construction is standard and is explained, for example, in \cite[pp. 3856-57]{GePa}. We can even do it through a conformal deformation of $g_0$, as explained in \cite[ pp. 718-719]{CEforms}. Therefore, we obtain a family of Riemannian manifolds $(M,g_\varepsilon)$ so that $M$ is the union of three parts 
$$
M=M_0 \cup C \cup S_0^n
$$
with

- $M_0$ is an open subset of $M$ and $g_\varepsilon$ does not vary with $\varepsilon$ on $M_0$,

- $(C, g_\varepsilon )$ is isometric to the cylinder $[0,\frac{1}{\varepsilon}]\times \Sp^{n-1}$ 
of length $\frac{1}{\varepsilon}$ (with $0 < \varepsilon \le 1$),

- $ S_0^n$ is a round hemisphere of  radius $1$  which closes the end of the cylinder $C$ and $g_\varepsilon\vert_{ S_0^n} $ is the round metric (and is independent of $\varepsilon$).

\smallskip
The only varying parameter in this construction is the length $\frac{1}{\varepsilon}$ of the cylinder $(C,g_{\varepsilon})$. Notice that 
the  volume of $(M,g_{\varepsilon})$ is not equal to 1, but we will make a suitable scaling at the end of the proof.

\smallskip
In order to bound the eigenvalues $\mu_k^{g_\varepsilon}(\rho, 1)$ from above, we will use the GNY method  \cite{GNY}. To this end,  we need   a uniform control (w.r.t. $\varepsilon$) of the {\it packing constant} (see \cite[Definition 3.3 and Theorem 3.5]{GNY}) and of the volume growth of balls in $(M,g_{\varepsilon})$. This will be done in the following lemmas. For this purpose, we introduce the connected open subset  $\tilde M_0\subset M$ obtained as the union of $M_0$ and  the part of the cylinder which corresponds to  $(0,3d_0)\times \Sp^{n-1}\subset [0,\frac{1}{\varepsilon}]\times \Sp^{n-1}$, where $d_0$ is the diameter of $M_0$.

\smallskip

\begin{lemme}[volume growth of balls]\label{volume}  
There exist two positive constants $C_1$ and $C_2$, independent of $\varepsilon$, such that, for every ball $B_\varepsilon(x,r) $ in $(M,g_{\varepsilon})$ we have
\begin{equation}\label{volball}
\vert B_\varepsilon(x,r)\vert_{g_\varepsilon} \le \left\{
 \begin{array}{lll}
C_1 r^n &  \   & \text{if}  \ r \le 2 d_0\\
    C_2r &  \   &  \text{if}  \  r \ge 2 d_0
\end{array}
\right.
\end{equation}

\end{lemme}

\begin{proof} 
If $B_\varepsilon(x,r)\cap M_0= \emptyset$, then $B_\varepsilon(x,r)$ is isometric to a geodesic ball of radius $r$ of the capped cylinder and an obvious calculation shows that \eqref{volball} holds true with two constants $C_1$ and $C_2$ independent of $\varepsilon$ (in fact, we can compare the volume of  $B_\varepsilon(x,r)$ with the volume of $ (-r,r)\times \Sp^{n-1}$ to get $\vert B_\varepsilon(x,r)\vert_{g_\varepsilon}\le A r$ for some positive $A$).
If $B_\varepsilon(x,r)\cap M_0\ne \emptyset$ and $r< 2d_0$, then $ B_\varepsilon(x,r)$ is contained in $\tilde M_0$. Hence, there exists a constant $C$, depending only on $\tilde M_0$, such that  $\vert B_\varepsilon(x,r)\vert_{g_\varepsilon}\le C r^n$.
 If $B_\varepsilon(x,r)\cap M_0\ne \emptyset$ and $r\ge 2d_0$, then 
 $B_\varepsilon(x,r)$ is contained in the union of a ball $B(x_0,2d_0)\subset \tilde M_0$ centered at a point $x_0\in M_0$ and a ball of radius $r'\le r$ contained in the cylindrical  part.   Thus, $\vert B_\varepsilon(x,r)\vert_{g_\varepsilon}\le C 2^n d_0^n + A r\le C_2r$ for some positive $C_2$ which does not depend on $\varepsilon$.
 \end{proof}
 
\begin{lemme}\label{packing}  
There exists a constant $N$, independent of $\varepsilon$, such that any ball of radius $r>0$ in $(M,g_{\varepsilon})$ can be covered by $N$ balls of radius $\frac r2$.
\end{lemme}
\begin{proof} Let $B_\varepsilon(x,r) $ be a ball of radius $r$  in $(M,g_{\varepsilon})$.  If $B_\varepsilon(x,r)\cap M_0= \emptyset$, then, since $(M\setminus M_0, g_\varepsilon)$ is isometric to    the capped cylinder   whose Ricci curvature is everywhere nonnegative,  $B_\varepsilon(x,r) $ can be covered by $N_E$ balls of radius $\frac r2$, where $N_E$ is the packing constant of the Euclidean space $\R^n$ (Bishop-Gromov theorem).

Assume that $B_\varepsilon(x,r) \cap M_0\ne \emptyset$.  If $r< 2d_0$, then $B_\varepsilon(x,r) $ is contained in $\tilde M_0$. Thus, $B_\varepsilon(x,r) $ can be covered by $N(\tilde M_0)$ balls of radius $\frac r2$, where $N(\tilde M_0)$ is the the packing constant of  $\tilde M_0$.
If $r\ge 2d_0$, then 
$B_\varepsilon(x,r) $ is contained in the union of a ball $B_\varepsilon(x_0,2d_0)\subset \tilde M_0$ centered at a point $x_0\in M_0$ and a ball of radius $r'\le r$ contained in the capped cylinder.  Again, $B_\varepsilon(x,r) $ can be covered by $N_E + N(\tilde M_0)$ balls of radius $\frac r2$.
 \end{proof}

\begin{proof}[Proof of Theorem \ref{inf_sup_rho}] Let $\rho$ be a positive density on $M$ with $\fint_M\rho v_{g_\varepsilon} =1$.
Applying  \cite[Theorem 3.5]{GNY} to the metric measured space $(M,d_\varepsilon, \rho v_{g_\varepsilon} )$, where  $d_\varepsilon$ is the Riemannian distance associated to $g_\varepsilon$, we deduce the existence of $k+1$  annuli $A_1,\dots, A_{k+1}$ such that $\int_{A_j} \rho v_{g_\varepsilon}\ge \frac{\vert M\vert _{g_\varepsilon}}{Ck}$ and $2A_1,\dots 2A_{k+1}$ are mutually disjoint. Here, $C$ should depends on   the packing constant of $(M,g_\varepsilon)$, but since the latter  is dominated independently of $\varepsilon$, thanks to Lemma  \ref{packing}, we can assume that $C$ is independent of $\varepsilon$.

To each annulus  of the form 
 $A=B_\varepsilon(x,R) \setminus B_\varepsilon(x,r) $ we associate a function  $u_A$ defined  as in   \eqref {testfunct}.  We obtain
$$
R_{(g_\varepsilon, \rho, 1)}(u_A) =\frac{\int_{2A} \vert \nabla^\varepsilon  u_A\vert_{g_\varepsilon}^2 v_{g_\varepsilon}}{\int_{2A}   u_A^2 v_{g_\varepsilon}}\le  \frac{\frac 4{r^2}{\vert B_\varepsilon(x,r)\vert_{g_\varepsilon}} +\frac 1{R^2}{\vert B_\varepsilon(x,2R)\vert_{g_\varepsilon}} }{\int_{A} \rho v_{g_\varepsilon}}
.$$
Using Lemma \ref{volume}  we get for every $r>0$,
\begin{equation}
\frac 1{r^2}\vert B_\varepsilon(x,r)\vert_{g_\varepsilon} \le \left\{
 \begin{array}{lll}
C_1 r^{n-2}\le C_1 d_0^{n-2}  &  \   & \text{if}  \ r \le 2 d_0\\
    \frac {C_2}r \le  \frac {C_2}{2d_0} &  \   &  \text{if}  \  r \ge 2 d_0
\end{array}
\right.
\end{equation}
Therefore, there exists a constant $ C'$ which depends on $C_1$, $C_2$ and $d_0$ (but independent of $\varepsilon$), such that
$$
R_{(g_\varepsilon, \rho, 1)}(u_A) \le \frac { C'}{\int_{A} \rho v_{g_\varepsilon}}
.$$
Consequently, the $k+1$ annuli $A_1,\dots, A_{k+1}$ provide $k+1$ disjointly supported functions satisfying $R_{(g_\varepsilon, \rho, 1)}(u_{A_j})\le \frac { C'}{\int_{A_j} \rho v_{g_\varepsilon}}\le \frac {CC' k}{\vert M\vert_{g_\varepsilon}} $. Thus, 
$$\mu_k^{g_\varepsilon}( \rho, 1)\le C'' \frac{k}{\vert M\vert _{g_\varepsilon}}.$$

In order to obtain a family of metrics  of volume 1 we set $g'_{\varepsilon} = \frac{1}{\vert M\vert _{g_\varepsilon}^{2/n}} g_\varepsilon$.   Hence, for any $\rho$ such that $\fint_M \rho \, v_{g'_\varepsilon}=\fint_M \rho \, v_{g_\varepsilon} =1$, we have
$$\mu_k^{ g'_\varepsilon}( \rho, 1)=  {\vert M\vert _{g_\varepsilon}^{2/n}} \mu_k^{g_\varepsilon}( \rho, 1)\le C'' \frac{k}{\vert M\vert _{g_\varepsilon}^{1-\frac 2n}}.$$
 But $\vert M\vert _{g_\varepsilon}\ge \vert C\vert _{g_\varepsilon}\ge  \frac {n\omega_n}\varepsilon$. Thus  
 $$\mu_k^*(M,g'_\varepsilon)\le   C {k}\varepsilon^{1-\frac 2n }.$$
 \end{proof}


\begin{proof}[Proof of Theorem \ref{inf_sup_sigma}] Let $(M,g_\varepsilon)$ be as in the construction above and let $\sigma$ be such that $\int_M \sigma v_{g_\varepsilon}=\vert M\vert_{g_\varepsilon}$.
The cylindrical part $(C,g_{\varepsilon})$ of $(M,g_\varepsilon)$ can be decomposed into $2(k+1)$ small cylinders  $C_j\approx [\frac{j}{2(k+1)\varepsilon},\frac{j+1}{2(k+1)\varepsilon}] \times \Sp^{n-1}$, $j=0,...,2k+1$, of length $\frac{1}{2(k+1)\varepsilon}$.
At least $(k+1)$ cylinders among $C_0, \dots, C_{2k+1}$ have a measure with respect to $\sigma  $ which is less or equal to $\frac{\vert M\vert _{g_\varepsilon}}{k+1}$. To each such   $C_j$  we associate a function $f$ with support in $C_j$ and which is defined in $C_j$, through the obvious identification  between $C_j$ and   $[0,\frac{1}{2(k+1)\varepsilon}]\times \Sp^{n-1}$, 
as follows: $\forall (t,z)\in [0,\frac{1}{2(k+1)\varepsilon}]\times \Sp^{n-1}\approx C_j$, 
\begin{equation}
f(t,z)= \left\{
 \begin{array}{lll}
6(k+1)\varepsilon t &  \   & \text{if} \ \  0 \le t\le \frac{1}{6(k+1)\varepsilon}\\
1 &  \   &  \text{if}  \ \ \frac{1}{6(k+1)\varepsilon} \le t\le \frac{2}{6(k+1)\varepsilon}\\
-6(k+1)\varepsilon t+3  &  \   & \text{if} \ \  \frac{2}{6(k+1)\varepsilon}\le t\le \frac{3}{6(k+1)\varepsilon}.
\end{array}
\right.
\end{equation}
We have
$$
\int_Mf^2  v_{g_\varepsilon}\ge \int_{ [\frac{1}{6(k+1)\varepsilon},\frac{2}{6(k+1)\varepsilon}]\times \Sp^{n-1}}f^2 \ v_E = \frac{n\omega_n}{6(k+1)\varepsilon}
$$
where $v_E$ is the standard product measure. On the other hand,
 the norm of the gradient of $f$ is supported in $C_j$ and is dominated by $6(k+1)\varepsilon$. Thus, 
 $$
\int_M \vert \nabla^\varepsilon f\vert_{g_\varepsilon}^2\sigma v_{g_\varepsilon}  \le  (6(k+1)\varepsilon)^2 \int_{C_j}  \sigma v_{g_\varepsilon} \le (6(k+1)\varepsilon)^2 \frac{\vert M\vert _{g_\varepsilon}}{k+1}= 36(k+1)\varepsilon^2\vert M\vert _{g_\varepsilon}
$$
and  the Rayleigh quotient of $f$ satisfies
$$
R_{(g_\varepsilon, 1,\sigma)}(f) \le \frac{216(k+1)^2\varepsilon^3 \vert M\vert _{g_\varepsilon}}{n\omega_n}.
$$
Consequently, the $k+1$ chosen cylinders  provide $k+1$ disjointly supported functions satisfying the last inequality, which yields
$$\mu_k^{g_\varepsilon}(1,\sigma) \le C \vert M\vert _{g_\varepsilon}(k+1)^2 \varepsilon^3 $$
with $C= \frac{216}{n\omega_n}.$
Setting $g'_{\varepsilon}= \frac1{\vert M\vert _{g_\varepsilon}^{\frac2n}}g_{\varepsilon}$, we get
$$\mu_k^{g'_\varepsilon}(1,\sigma)=\vert M\vert _{g_\varepsilon}^{\frac2n} \mu_k^{g_\varepsilon}(1,\sigma) \le C \varepsilon^3 \vert M\vert _{g_\varepsilon}^{1+\frac 2n}(k+1)^2$$
with $\vert M\vert _{g_\varepsilon}= \vert \tilde M_0\vert _{g} +\vert C \vert _{g_\varepsilon} +\frac 12 n\omega_n\le \frac A\varepsilon$ for some constant $A$. 
 Thus
$$\mu_k^{**}(M,g'_\varepsilon)\le   C'\varepsilon^{2-\frac2n} (k+1)^2.$$
\end{proof}

\begin{rem}
The same type of construction used in the proof of Theorems \ref{inf_sup_rho}  and \ref{inf_sup_sigma} allows us to prove the existence of a family of bounded domains $\Omega_\varepsilon\subset\R^n$  of volume 1 such that $\mu_k^*(\Omega_\varepsilon,g_E) $ (resp. $\mu_k^{**}(\Omega_\varepsilon,g_E) $)  goes to zero with $\varepsilon$. This is to be compared with the result of Proposition \ref{conformaleucdomain}.
\end{rem}

We end this section with the following proposition in which we show how to produce   examples of manifolds $(M,g_\varepsilon)$  of fixed volume for which the ratio $\frac{\mu_1^*(M,g_\varepsilon)}{\lambda_1(M,g_\varepsilon)}$ (resp. $\frac{\mu_1^{**}(M,g_\varepsilon)}{\lambda_1(M,g_\varepsilon)}$)  tends to infinity as $\varepsilon\to 0$.  

\begin{prop}\label{small_big}
Let $M$ be a compact manifold and let $A$ be a positive constant. 

\noindent(i) There exists  a family of metrics $g_\varepsilon$ {\it of volume 1}  on $M$ and a constant $A>0$ such that  $\forall \varepsilon \in (0,1)$, $\lambda_1(M,g_\varepsilon)\le \varepsilon$ while $\mu_1^*(M,g_\varepsilon)\ge A$.

\smallskip
\noindent(ii) There exists  a family of metrics $g_\varepsilon$ {\it of volume 1 } on $M$  and a constant $A>0$ such that, $\forall \varepsilon \in (0,1)$, $\lambda_1(M,g_\varepsilon)\to 0$ while $\mu_1^{**}(M,g_\varepsilon)\ge A$.

\end{prop}

\begin{proof}

 (i) Let us start with a Riemannian metric $g$ of volume one on $M$  such that an open set $V$ of $M$ is isometric to the Euclidean ball of volume  $\frac 12$. By a standard argument (Cheeger Dumbbell construction), one can deform the metric $g$ outside  $V$ in a metric $g_\varepsilon$ of volume 1 such that $\lambda_1(M,g_\varepsilon)\le \varepsilon$.   Applying Corollary \ref{density_neumann} with $M_0=V$, we get  $\mu_1^*(M, g_\varepsilon)\ge \vert V\vert_{g_\varepsilon} \lambda_1(V,g_\varepsilon)=\frac 12 \lambda_1(V,g)$. Since $ \lambda_1(V,g) = (2\omega_n)^{\frac 2n} \lambda_1(B^n,g_E)$, where $B^n$ is the unit Euclidean ball, we get the desired inequality with $A= \frac12(2\omega_n)^{\frac 2n} \lambda_1(B^n,g_E)$.

 
 \smallskip
\noindent (ii) Let   $g$ be a Riemannian metric on $M$ such that an open subset $V$ of $M$ is isometric to the  capped cylinder $C=(-2,2)\times \Sp^{n-1}$ closed by a spherical cap. 
We will deform the metric $g$ inside $V$ so that  $(M,g_{\varepsilon})$ looks like a Cheeger dumbbell (thus $\lambda_1(M,g_{\varepsilon})\to 0$ as $\varepsilon \to 0$) and associate to $g_{\varepsilon}$ a family of densities such that   $\mu_1^{g_{\varepsilon}}(1,\sigma_\varepsilon) \ge A>0$. Indeed, 
the metric on the cylinder $C=(-2,2)\times \Sp^{n-1}$ is given in coordinates  
$(t,x)\in (-2 ,  2)\times \Sp^{n-1}$ by $g_\varepsilon(t,x)=dt^2+ \gamma_{\varepsilon}^2(t) g_{\Sp^{n-1}}$ with $\gamma_{\varepsilon} (-t)=\gamma_{\varepsilon} (t)$ and

\begin{equation}
\gamma_\varepsilon (t)= \left\{
 \begin{array}{lll}
   \varepsilon &  \   &  \text{if}  \  t\in      [0  ,  \frac{1} {2}]\\
    \in (\varepsilon, 1) &  \   &  \text{if}  \  t\in      [\frac{1}{2},1 ]\\
    1 &  \   & \text{if}  \ t\in [1  ,  2)
\end{array}
\right.
\end{equation}
We do not change the metric $g$ outside $V$. We endow $(M,g_{\varepsilon})$ with  the density $ \sigma_\varepsilon $ given by  $\sigma_\varepsilon (t,x)=\frac {1}{\gamma_{\varepsilon} (t)^{n-1}}$ on the cylinder $C$ and extended  by $1$ outside $C$.

\smallskip
It is well known that $\lambda_1(M,g_{\varepsilon}) \to 0$ as $\varepsilon \to 0$.
Let us study $\mu_1^{g_{\varepsilon}}(1,\sigma_{\varepsilon})$. One has  for every  
 $f\in C^\infty (M)$ 
$$\int_M \vert \nabla^\varepsilon f\vert_{g_\varepsilon}^2 \sigma_\varepsilon v_{g_\varepsilon} =
\int_{M\setminus C} \vert \nabla f\vert_{g}^{2} v_{g}+  \int_{-2}^{2}dt \int_{\Sp^{n-1}} \vert \nabla^\varepsilon f\vert_{g_\varepsilon}^2 \sigma_\varepsilon(t)\gamma_\varepsilon (t)^{n-1} v_{\Sp^{n-1}}$$
$$  \qquad   =\int_{M\setminus C} \vert \nabla f\vert_{g}^{2} v_{g}+  \int_{-2}^{2}dt \int_{\Sp^{n-1}} \vert \nabla^\varepsilon f\vert_{g_\varepsilon}^2 v_{\Sp^{n-1}}$$
where $v_{\Sp^{n-1}}$ denotes the volume form on the sphere $\Sp^{n-1}$.
Now, observe that  $ \vert\nabla^\varepsilon f\vert_{g_\varepsilon}^2$ can be estimated as follows:
$$ \vert\nabla^\varepsilon f\vert_{g_\varepsilon}^2= \left(\frac{\partial f}{\partial t}\right)^2+\vert \nabla_{0} f\vert^2 \gamma_{\varepsilon}(t)^{-2} \ge \left(\frac{\partial f}{\partial t}\right)^2+\vert \nabla_{0} f\vert^2 =  \vert \nabla f\vert_{g}^2$$
where $\nabla_0f$ is the tangential part of the gradient of $f$ w.r.t.   $\Sp^{n-1}$. Therefore,
$$\int_M \vert \nabla^\varepsilon f\vert_{g_\varepsilon}^2 \sigma_\varepsilon v_{g_\varepsilon}  \ge \int_{M\setminus C} \vert \nabla f\vert_{g}^{2} v_{g}+\int_{-2}^{2}dt \int_{\Sp^{n-1}}  \vert \nabla f\vert_{g}^2  v_{\Sp^{n-1}} = \int_M \vert \nabla f\vert_{g}^2  v_{g}.$$
On the other hand (since $\gamma_\varepsilon (t)^{2}\le 1$)
$$\int_M f^2  v_{g_\varepsilon} \le \int_{M }f^2 v_{g}.$$
In conclusion, for every $f\in C^\infty (M)$, one has 
$$R_{(g_{\varepsilon},1,\sigma_{\varepsilon})}(f)\ge R_{(g,1,1)}(f).
$$
It follows, thanks to the min-max principle, that 
$$\mu_1^{g_{\varepsilon}}(1,\sigma_{\varepsilon})\ge \lambda_1(M,g).$$ 
The last point is to suitably rescale $g_{\varepsilon}$ and $\sigma_\varepsilon$. For this purpose, just observe that $\int_M \sigma_{\varepsilon}v_{g_{\varepsilon}} =\vert M\vert_g$ and $\frac 12 \vert M\vert_g\le \vert M\vert_{g_\varepsilon}\le \vert M\vert_g$.
\end{proof}

 \section{Examples } \label{examples}
 
 In this section we describe  situations in which we can compute or give explicit estimates for the first extremal  eigenvalues. 
Let $(M,g) $ be a compact Riemannian manifold of dimension $n\ge 2$, possibly with a nonempty boundary. 
 
 \begin{prop}\label{confsphere} Assume that there exists a conformal map $\phi$ from $(M,g)$ to the standard $n$-dimensional sphere $\Sp^n$. Then, 
 \begin{equation}\label{confsphere1}
\lambda_1^c(M,g)=n  {\alpha_n}^{\frac2n}
\end{equation}
and
 \begin{equation}\label{confsphere2}
 \mu_1^*(M,g)\le n  \left(\frac{\alpha_n}{|M|_g}\right)^{\frac2n}
 \end{equation}
where   $\alpha_n$ is the volume of the unit Euclidean $n$-sphere. Moreover, if $n=2$, then the equaliy holds in \eqref{confsphere2}.
  \end{prop}
Notice that when $(M,g)$ is the standard sphere $\Sp^n$, then the equaliy holds in \eqref{confsphere2} (see Corollary \ref{homogene-irred} below).

 \begin{proof}[Proof of Propositon \ref{confsphere}]   
Let us first  prove   \eqref{confsphere2}.
 Let $\rho$ be a density on $M$ with $\int_M\rho v_g=1$.
Given any nonconstant map $\phi=(\phi_1,\cdots,\phi_{n+1}):(M,g)\to \Sp^n$,  a standard argument tells us that there exists a conformal diffeomorphism ${\gamma\in Conf(\Sp^n)}$ such that $\psi=\gamma\circ \phi$ satisfies $\int_M \psi_j \rho \, v_g =0$, $j=1\dots, n+1$  (see for instance \cite[Proposition 4.1.5]{GNP}).
Thus, $\forall j\le n+1$,
$$\mu_1( \rho,1)\int_M\psi_j^2 \rho \, v_g \le\int_M \vert \nabla\psi_j \vert ^2 v_g $$
(see \eqref{intro0}) and, summing up w.r.t. $j$, 
$$\mu_1( \rho,1)\int_M\rho \, v_g \le\int_M \vert d\psi\vert ^2 v_g \le \left(\int_M \vert d\psi\vert ^n v_g\right)^{\frac2n} \vert M\vert_g^{1-\frac2n}.$$
Since $\psi=\gamma\circ\phi$ is a conformal map,  $\int_M \vert d\psi\vert ^n v_g$ is nothing but $n^{\frac n2}$ times the volume of $\psi(M)\subset\Sp^n$ with respect to the standard metric $g_s$ of $\Sp^n$ (indeed, $\psi^*g_s=\frac 1n \vert d\psi\vert ^2g$). Therefore, 
$$\mu_1( \rho,1)\fint_M\rho v_g\le n\vert \psi(M)\vert_{g_s}^{\frac 2n} \vert M\vert_g^{-\frac2n}\le n  \left(\frac{\alpha_n}{|M|_g}\right)^{\frac2n}$$
which proves  \eqref{confsphere2}.

 Using the same arguments we can prove the inequality $\lambda_1^c(M,g)\le n  {\alpha_n}^{\frac2n}$. The reverse inequality follows from \cite[Theorem A]{CE}. 
   \end{proof}  
   
   It is well known that the Euclidean space $\R^n$ and the hyperbolic space $\mathbb H^n$  are conformally equivalent to open parts of the sphere $\Sp^n$. This leads to the following corollary.  
   
   \begin{cor}\label{mu1sup} Let $\Omega$ be a  bounded domain of the Euclidean space $\R^n$, the hyperbolic space $\mathbb H^n$ or the sphere $\Sp^n$, endowed with the induced metric $g_s$. One has
$$\lambda_1^c(\Omega,g_{s})=n  {\alpha_n}^{\frac2n}$$
and
$$\mu_1^*(\Omega,g_{s})\le n  \left(\frac{\alpha_n}{|\Omega|}\right)^{\frac2n}.$$
Moreover,  the following equality holds in   dimension 2: $\mu_1^*(\Omega,g_{s})=\lambda_1^c(\Omega,g_{s}){|\Omega|}^{-1}=\frac{8\pi }{|\Omega|}.$
\end{cor}

\begin{rem}
Let $D$ be the unit disc in $\R^2$ and let $\rho_t=\frac{4t}{(t^2\vert z\vert^2+1)^2} $. Then 
$$\mu_1^*(D,g_{E}) =\lim_{t\to\infty}\mu_1^{g_{E}}(\frac{\rho_t}{\fint_D\rho_t dx},1) = 8.$$
Indeed, the map $\phi_t(z)= \frac 1{t^2\vert z\vert^2+1}(2 tz,   t^2\vert z\vert^2-1)$ identifies $(D, \frac{4t}{(t^2\vert z\vert^2+1)^2} g_E)$ with  a spherical cap $C_t$ in $\Sp^2$ whose radius goes to $\pi$ as $t\to\infty$. Hence,
$\mu_1^{g_{E}}(\rho_t,1)\int_D\rho_t dx = \mu_1(C_t)\vert C_t\vert$ which converges to $8\pi$ as $t\to\infty$.
 
\end{rem}

    \begin{prop}\label{constantenergy} Assume that there exists a map $\phi:(M,g)\to\Sp^p$ from $(M,g)$ to the standard $p$-dimensional sphere $\Sp^p$ satisfying both $\int_M\phi v_g=0$ and $\vert d\phi\vert^2\le\Lambda$ for some positive constant  $\Lambda$. Then
 \begin{equation}\label{constantenergy1}
    \mu_1^{**}(M,g)\le \Lambda.
\end{equation}

  \end{prop}

 \begin{proof} 
One has, for every $j\le p+1$, 
$$\mu_1( 1,\sigma)\int_M\phi_j^2  \, v_g \le\int_M \vert \nabla\phi_j \vert ^2\sigma v_g $$
and, summing up w.r.t. $j$, 
$$\mu_1( 1,\sigma)\vert M\vert_g \le\int_M \vert d\phi\vert ^2 \sigma v_g \le \Lambda \int_M  \sigma v_g$$
which implies \eqref{constantenergy1}.
   \end{proof}  
   
If   $(M,g)$  be a compact homogeneous Riemannian manifold, and if $\phi_1,\dots,\phi_p$ is an $L^2$-orthonormal  basis of the first eigenspace of the Laplacian, then both $\sum_{i\le p} \phi_i^2$ and $ \vert d\phi \vert ^2=\sum_{i\le p} \vert d\phi_i\vert^2$ are constant on $M$. This enables us to apply Proposition \ref{constantenergy} and get the following
   \begin{cor} Let $(M,g)$  be a compact homogeneous Riemannian manifold. Then 
  $$ \mu_1^{**}(M,g)=\mu_1(M,g)$$
\end{cor}
In other words, on  a compact homogeneous Riemannian manifold, $\mu_1(1,\sigma)$ is maximized when $\sigma $ is constant. 

\begin{ex}In \cite{EIR}, it is proved that 
if
$\Gamma={\mathbb Z}e_1+ {\mathbb Z}e_2 \subset {\mathbb R}^2$ is a
lattice such that $|e_1 | =|e_2|$, then the corresponding flat metric
$g_{_\Gamma}$ on the torus ${\mathbb T}^2$ satisfies $ \mu_1 ^c ({\mathbb
T}^2, g_{_\Gamma})=\lambda_1 ({\mathbb
T}^2,g_{_\Gamma})\vert {\mathbb T}^2 \vert_{g_{_\Gamma}}$. A higher dimensional
version of this result was also established in \cite{EI4}. Since a flat Torus is a 2-dimensional homogeneous Riemannian manifold, we have the following equalities
  $$\lambda_1^c ({\mathbb
T}^2, g_{_\Gamma})\vert {\mathbb T}^2 \vert_{g_{_\Gamma}}^{-1}
=\mu_1^{*}({\mathbb T}^2, g_{_\Gamma})=\mu_1^{**}({\mathbb T}^2, g_{_\Gamma})=\lambda_1  ({\mathbb
T}^2, g_{_\Gamma}).$$
Neverthless, whereas we always have $\mu_1^{**}({\mathbb T}^2, g_{_\Gamma})=\mu_1  ({\mathbb
T}^2, g_{_\Gamma})$, it follows from \cite[Theorem A]{CE} that when the length ratio $|e_2| /
|e_1 |$ of the vectors $e_1$ and $e_2$ is sufficiently far from 1,
then $\mu_1^{*}({\mathbb T}^2, g_{_\Gamma})= \lambda_1 ^c ({\mathbb
T}^2, g_{_\Gamma})\vert {\mathbb T}^2 \vert_{g_{_\Gamma}}^{-1}>\lambda_1 ({\mathbb
T}^2,g_{_\Gamma})$. 
\end{ex}

Recall that a map $\phi=(\phi_1,\cdots,\phi_{p+1}):(M,g)\to \Sp^p$ is harmonic if and only if its components $ \phi_1,\cdots,\phi_{p+1}$ satisfy
$$\Delta_g\phi_j = -\vert d\phi\vert ^2  \phi_j, \quad j=1\cdots, p+1.$$
The stress-energy tensor of a map $\phi$ is a symmetric covariant 2-tensor  defined for every tangent vectorfield $X$ on $M$ by: $S_\phi (X,X)=\frac 12 \vert d\phi\vert ^2 \vert X\vert^2_g-\vert d\phi (X)\vert^2$. In \cite[Theorem 3.1]{E} it is proved that  if the stress-energy tensor of a harmonic map $\phi$ is nonnegative, then, for every conformal diffeomorphism $\gamma$ of the sphere  $\Sp^p$ one has
$$\int_M \vert d(\gamma\circ\phi)\vert ^2 v_g \le \int_M \vert d\phi\vert ^2 v_g.$$
Moreover, the strict inequality holds if $\gamma$ is not an isometry and if  $S_\phi$ is positive definite at some point.  
Observe that if $\phi:(M,g)\to \Sp^p$ is a conformal map or a horizontally conformal map, then  $S_\phi$ is nonnegative (see \cite{E}).

  \begin{prop}\label{harmonicmap} Assume that there exists a harmonic map $\phi:(M,g)\to\Sp^p$  with nonnegative stress-energy  tensor. Then, 
 \begin{equation}\label{harmonicmap1}
 \mu_1^*(M,g)\le\fint_M \vert d\phi\vert^2 v_g.
  \end{equation}

  \end{prop}

 \begin{proof} 
 
 Let $\rho$ be a positive density on $M$. 
As before, we know that there exists ${\gamma\in Conf(\Sp^n)}$ such that $\psi=\gamma\circ \phi$ satisfies $\int_M \psi_j \rho \, v_g =0$, $j=1\dots, n+1$.
Thus
$$\mu_1( \rho,1)\int_M\psi_j^2 \rho \, v_g \le\int_M \vert \nabla\psi_j \vert ^2 v_g $$
and, summing up w.r.t. $j$, 
$$\mu_1( \rho,1)\int_M\rho \, v_g \le\int_M \vert d(\gamma\circ\phi)\vert ^2 v_g \le  \int_M \vert d\phi\vert ^2 v_g$$
which implies  \eqref{harmonicmap1}.
   \end{proof}  
   
  A  particular  case of Proposition \ref{harmonicmap} is when there exists a harmonic map $\phi:(M,g)\to\Sp^p$ which  is homothetic. In this case, $S_\phi =\frac{n-2}n \vert d\phi\vert^2 g$ and  $\vert d\phi\vert^2$ is constant  and coincides with an eigenvalue $\lambda_k(M,g)$ for some $k\ge 1$.  
  For example,  if $(M,g)$ is a compact isotropy irreducible homogeneous space (e.g. a compact rank-one symmetric space) and if $\phi_1,\dots,\phi_p$ is an $L^2$-orthonormal  basis of the first eigenspace of the Laplacian, then $\phi=\left(\frac{\vert M\vert_g}p\right)^{\frac 12}(\phi_1,\dots,\phi_p)$ is a harmonic map from $(M,g)$ to $\Sp^p$ which is homothetic and satisfies $\vert d\phi\vert^2=\lambda_1(M,g)$. Proposition \ref{harmonicmap} then implies that $\mu_1^{*}(M,g)=\lambda_1(M,g)$. On the other hand, 
   the second author and Ilias
\cite{EI} proved that  in this situation we also have $ \lambda_1^c (M,
g)= \lambda_1 (M,g)\vert M\vert_g^{\frac 2n}$. Consequently, we have the following

   \begin{cor}\label{homogene-irred} Let $(M,g)$  be a  compact isotropy irreducible homogeneous space. Then 
  $$ \lambda_1^{c}(M,g)\vert M\vert_g^{-\frac 2n}=\mu_1^{*}(M,g)=\mu_1^{**}(M,g)=\lambda_1(M,g).$$
\end{cor}

\def\cprime{$'$} \def\cprime{$'$} \def\cprime{$'$}


\end{document}